\documentclass[11pt]{article}
\usepackage{amssymb}
\usepackage{amsmath}
\usepackage{amsthm}
\usepackage{algorithm,algpseudocode}
\usepackage{graphicx}
\usepackage[backref]{hyperref} 
\usepackage{subcaption}
\usepackage{svg}
\usepackage[nohead,margin=1.0in]{geometry}
\usepackage[title]{appendix}
\usepackage{verbatim}
\newtheorem{theorem}{Theorem}[section]

\newtheorem{lemma}[theorem]{Lemma}

\numberwithin{equation}{section}

\title{Solving Inverse Obstacle Scattering Problem with Latent Surface Representations\thanks{The work of J. Chen is supported by National Key R\&D Program of China 2019YFA0709600, 2019YFA0709602. The work of B. Jin is supported by Hong
Kong General Research Fund (Project No. 14306423), and a start-up fund from The Chinese University of Hong Kong. }}
\author{Junqing Chen\thanks{Department of Mathematical Sciences, Tsinghua University, Beijing 100084, P.R. China.  (\texttt{jqchen@tsinghua.edu.cn}, \texttt{liuhb19@mails.tsinghua.edu.cn}).
} \and Bangti Jin\thanks{Deaprtment of Mathematics, The Chinese University of Hong Kong, Shatin, New Territories, Hong Kong SAR, P.R. China (\texttt{b.jin@cuhk.edu.hk}, \texttt{bangti.jin@gmail.com})} \and Haibo Liu\footnotemark[2]}

\begin{document}

\maketitle

\begin{abstract}
We propose a novel iterative numerical method to solve the three-dimensional inverse obstacle scattering problem of recovering the shape of an obstacle from far-field measurements. To address the inherent ill-posed nature of the inverse problem, we advocate the use of a trained latent representation of surfaces as the generative prior. This prior enjoys excellent expressivity within the given class of shapes, and meanwhile, the latent dimensionality is low, which greatly facilitates the computation. Thus, the admissible manifold of surfaces is realistic and the resulting optimization problem is less ill-posed. We employ the shape derivative to evolve the latent surface representation, by minimizing the loss, and we provide a local convergence analysis of a gradient descent type algorithm to a stationary point of the loss. We present several numerical examples, including also backscattered and phaseless data, to showcase the effectiveness of the proposed algorithm.
\end{abstract}


\section{Introduction}

Inverse scattering problems are concerned with determining the nature of an unknown scattering phenomenon (e.g., shape, position, size, and medium properties) from the knowledge (measurement) of the acoustic/electromagnetic/elastic scattered fields. Inverse scattering represents an important class of mathematical problems arising in science and engineering and has found many important real-world applications, e.g., remote sensing (radar and sonar), nondestructive evaluation, geophysical exploration, security check, through-wall imaging, biomedical imaging, and diagnosis \cite{colton1998inverse}. This work focuses on the inverse obstacle scattering problem (ISP), of recovering the scatterer support from sparse and noisy far-field measurements. 

Due to its broad range of applications, many numerical algorithms have been proposed to solve the ISP over the past several decades. Roughly they can be divided into two categories, i.e., iterative methods and direct methods. Iterative methods typically first formulate ISP into PDE-constrained optimization problems via variational regularization \cite{ItoJin:2015}, representing the surface parametrically or nonparametrically (e.g., level set method or phase-field approximation) \cite{audibert2022accelerated,dorn2006level}. The resulting optimization problems are then minimized via an iterative algorithm, e.g., gradient descent \cite{audibert2022accelerated} and Newton method \cite{hettlich1999second,Hohage:1999}. 
Iterative methods can often yield accurate reconstructions but are computationally demanding, due to the need of repeatedly solving the direct and adjoint problems. In sharp contrast,  direct methods, e.g., linear sampling method \cite{ColtonKirsch:1996}, reverse time migration \cite{Chen_2013, Chen_2013_1}, factorization method, multiple signal classification (MUSIC) \cite{kirsch2002music}, direct sampling method (DSM) \cite{ito2013direct,LiLiLiuLiu:2015}, and topological derivative analysis \cite{AmmariGarnier:2012,feijoo2004new}, aim at providing an estimate of the scatterer support non-iteratively, e.g., via the probing function in the direct sampling method; see the monographs \cite{CakoniColton:2006, KirschGrinberg:2008} for the mathematical theory for these methods. Generally, these methods often provide only rough estimates of the scatterer support but can be done much more efficiently than iterative methods. One can also combine direct methods with iterative ones, e.g., using the rough estimate from the direct sampling method as the initial guess in an iterative method \cite{ItoJinZhou:2013,li2020extended}.  The reconstruction algorithm proposed in this work belongs to the first class. 

While being highly flexible and capable of producing accurate results, existing iterative methods are somehow still very involved when applied to practical scenarios due to the ill-posed nature of ISP, especially in the three-dimensional case. For example, for level set based methods, solving related evolution equations is computationally intensive, and regularization has to be incorporated carefully so as to control the complexity of the surface to remedy the ill-posedness of ISP; existing studies have explored sparsity regularization \cite{winters2009sparsity}, regularized sampling method  \cite{colton2000regularized} and topological prior \cite{carpio2020bayesian}. Nonetheless, there remains a great demand for further developing iterative techniques with improved accuracy and faster convergence.

In this work, we aim to explore the potential of using deep neural networks (DNNs) as a prior constraint, to address the inherent ill-posed nature of ISP. We develop a novel reconstruction algorithm based on the latent surface representation. The latent representation, called deepSDF (deep signed distance function), is of relatively low dimensionality and learned from the training data (e.g., point clouds) in a supervised manner. Indeed modern generative models, e.g., variational autoencoder \cite{kingma2013auto}, generative adversarial network \cite{goodfellow2014generative} and score-based diffusion models \cite{song2020score}, allow effectively extracting a high-fidelity representation of the manifold of shapes from the training data that can subsequently be used as a powerful prior in an iterative reconstruction scheme. We combine the latent surface representation with shape derivative, to efficiently evolve the surface via any stand-alone optimization method. We analyze the local convergence behavior of a gradient descent type scheme (i.e. ADAM \cite{KingmaBa:2015}) to a stationary point of the objective functional, under certain assumptions on the problem setting. We present several numerical illustrations, including backscattering and phaseless data, to show the effectiveness of the approach. Numerically we observe that the method is highly efficient and enjoys outstanding robustness with respect to the data noise. Indeed, the algorithm converges within tens of iterations, and remarkably, it can still yield reasonable approximations for up to 40\% noise in the data. In sum, there are several distinct features of the proposed approach. First, it can avoid the tedious computation of evolution equations when compared with level set-based methods. Second, the strong learned prior significantly reduces the dimensionality of the variable to be optimized. This significantly speeds up the convergence of the optimization algorithm. Third and last, the latent representation is highly expressive for a class of complex surfaces, e.g., planes and geological structures, which are challenging for application-agnostic priors, e.g., sparsity and total variation, to capture. This makes the prior ideally suited to the specific application scenarios of ISP.

Last let us put this work into the context of solving ISP using deep learning. In the last few years, deep learning has greatly impacted ISP. Motivated by direct imaging methods, the work \cite{wei2018deep} proposed three effective supervised methods, i.e., direct inversion scheme, back-propagation scheme, and dominant current scheme using U-Net \cite{RonnebergerFischer:2015}. Xu et al \cite{xu2020deep} proposed a deep learning method for the inversion of ISP with different schemes for the input of the neural network. Khoo and Ying \cite{khoo2019switchnet} proposed a novel neural network architecture, termed SwitchNet, to solve the forward and inverse scattering problems, by building mathematical insights into the architecture design. Likewise, Fan and Ying \cite{fan2019solving} proposed to interpolate Green's function with a neural network, utilizing especially the hierarchical low-rank structure of Green's function. Inspired by \cite{guo2021construct}, Ning et al \cite{NingHanZou:2023} proposed using DNNs to enhance direct sampling methods to solve ISP. We refer interested readers to \cite{chen2020review} for an overview of deep learning approaches to inverse scattering problems. All the aforementioned approaches are supervised in nature, and require abundant paired training data, in order for them to perform well.  Thus they differ markedly from the present work which uses the latent surface representation as a prior. The work \cite{guillard2021deepmesh} has used the latent representation to solve shape optimization problems, but the forward problem is modeled with a neural network. This makes heavy regularization crucial, and thus the optimized shape can only differ from the initial guess slightly.

The rest of the paper is organized as follows. In Section \ref{problem}, we describe the mathematical formulation of the inverse scattering problem. Then in Section \ref{method}, we present the DeepSDF surface representation and develop the algorithm for the ISP. In Section \ref{analysis}, we discuss one related theoretical issue, i.e., convergence guarantee of a gradient type method to a stationary point of the loss. Finally, in Section \ref{results} we present extensive numerical experiments to illustrate the effectiveness of the approach. We conclude the paper with further remarks in Section \ref{conclusion}.

\section{Inverse obstacle scattering problem}\label{problem}
Let $\Omega \subset \mathbb{R}^3$ be an open and bounded domain with a $C^2$ boundary $\Gamma$. Let $u^i$ be the incident plane wave, given by 
\begin{eqnarray}\label{incident}
    u^i(x)=e^{ikx\cdot d}, \quad x\in \mathbb{R}^3,
\end{eqnarray}
where $k > 0$ is the wave number and $d \in  \mathbb{S}^2$ the incident direction. Consider the following Helmholtz equation with a Dirichlet boundary condition
\begin{subequations}\label{hel}
    \begin{align}
        \Delta u + k^2 u  &=0,  \quad \mbox{in }\mathbb{R}^3 \backslash \Omega ,\label{hel1}\\
        u &= 0, \quad \mbox{on } \Gamma, \label{hel2}\\
        \frac{\partial u^s}{\partial r}-iku^s & =o(r^{-1}),\quad \mbox{as }r\to\infty,\label{hel3}
    \end{align}
\end{subequations}
where $r=\|x\|_2$ ($\|\cdot\|_2$ denotes the Euclidean norm of a vector), $u^s$ is the scattered field, and $u=u^i+u^s$ the total field. In the system \eqref{hel}, \eqref{hel1} is the Helmholtz equation, \eqref{hel2} the sound-soft (Dirichlet) boundary condition, and
\eqref{hel3} the Sommerfeld radiation condition. It is well known that the system \eqref{hel} has a unique solution \cite{nedelec2001acoustic}. The
radiation condition implies the following asymptotic behavior at infinity
\begin{eqnarray*}
    u(x)=\frac{e^{ikr}}{r}\Big\{u^\infty(\hat{x},d)+  \mathcal{O}\Big(\frac{1}{r}\Big)\Big\}, \quad \mbox{as } r\rightarrow \infty,
\end{eqnarray*}
where $\hat{x}=x/\|x\|_2$ is the direction of $x$, $u^\infty(\hat{x},d)$ is known as the far field pattern of $u$. It is well known that $u^\infty$ is an analytic function on the unit sphere $\mathbb{S}^2$.

Below we employ an integral equation formulation of the scattering problem \eqref{hel}. Recall that the fundamental solution $\Phi (x, y)$ of the Helmholtz equation is given by
\begin{eqnarray*}
    \Phi(x,y)=\frac{1}{4\pi}\frac{e^{ik\|x-y\|_2}}{\|x-y\|_2}, \quad x\neq y \in \mathbb{R}^3.
\end{eqnarray*}
Then we define the associated single-layer and double-layer potential operators $\mathcal{S}$ and $\mathcal{K}$ by
\begin{eqnarray*}
    (\mathcal{S}\varphi)(x)=\int_{\Gamma}\Phi(x,y)\varphi(y){\rm d}s(y)\quad \mbox{and}\quad
    (\mathcal{K}\varphi)(x)=\int_{\Gamma}\frac{\partial \Phi(x,y)}{\partial \nu(y)}\varphi(y){\rm d}s(y),
\end{eqnarray*}
respectively, and also the normal derivative operator $\mathcal{K}'$ by
\begin{eqnarray*}
    (\mathcal{K}'\varphi)(x)=\int_{\Gamma}\frac{\partial \Phi(x,y)}{\partial \nu(x)}\varphi(y){\rm d}s(y).
\end{eqnarray*}
Since the plane incident wave $u^i$ in \eqref{incident} satisfies the Helmholtz equation in the whole $\mathbb{R}^3$ and the total field $u$ and $u^s$ satisfy the exterior Helmholtz problem, we deduce that the normal derivative $v:=\partial u / \partial \nu$ satisfies the so-called  Burton-Miller combined boundary integral equations \cite[p. 59]{colton1998inverse}
\begin{eqnarray}\label{bie}
    v+\mathcal{K}^\prime
    v-ik\mathcal{S}v=2\frac{\partial u^i}{\partial\nu}-2iku^i,\quad\mbox{on }\Gamma.
\end{eqnarray}
Equation \eqref{bie} lays the foundation of a numerical treatment of problem \eqref{hel} via the boundary element method. Using  $v$, the scattered field $u^s$ on the exterior region $\mathbb{R}^3 \backslash \Omega$ is given by
\begin{eqnarray*}
    u^s(x)=-\frac{1}{4\pi}\int_{\Gamma} \Phi(x,y) v(y) {\rm d}s(y), \quad x \in \mathbb{R}^3 \backslash \Omega.
\end{eqnarray*}
The far filed pattern $u^\infty_{\Omega}(\hat{x},d)$ admits the following integral representation \cite{colton1998inverse}:
\begin{equation}
    u^\infty_{\Omega}(\hat{x},d)=
    \frac{1}{4\pi}\int_{\Gamma} \frac{\partial e^{-ik\hat{x}\cdot y}}{\partial \nu}(y)u^s(y) - \frac{\partial u^s}{\partial \nu}(y)e^{-ik\hat{x}\cdot y} {\rm d}s(y),\label{ffield} 
\end{equation}
where the subscript $\Omega$ indicates the dependence of  $u^\infty$ on the domain $\Omega$. ISP aims to reconstruct the obstacle boundary $\Gamma$, from a knowledge of data related to the far field pattern $u^\infty$. Throughout it is assumed that we can access the far field pattern at several measurement directions $\hat{x}_m, m = 1,...,M$, and for several incident plane waves \eqref{incident} with directions $d_l,l = 1,...,L$. 

Following the standard output least-squares formulation, one potential reconstruction algorithm \cite{colton1998inverse} for the ISP can be described as a shape optimization problem that looks for a minimizer $\Omega$ over the set of smooth domains, using the cost functional
\begin{equation}\label{loss}
    J(\Gamma)=\frac{1}{2LM} \sum_{l=1}^L \sum_{m=1}^M|u^\infty_{\Omega}(\hat{x}_m,d_l)-u^\infty_{\Omega^\star}(\hat{x}_m,d_l)|^2,
\end{equation}
where $\Omega^\ast$ denotes the true obstacle. 
This work builds on the output least-squares formulation but with one key contribution, of representing the domain $\Omega$ implicitly via a deep latent surface representation.

\section{Reconstruction of the surface via latent representation}\label{method}

\subsection{Latently represented surfaces}
To solve 3D ISP, the representation of the surface is key to the success of any imaging algorithm. In existing studies, implicit representation methods are dominant. A 3D Surface $S$ can be implicitly represented by the signed distance function $SDF_S(x):\mathbb{R}^3 \rightarrow \mathbb{R}$, which is a continuous function that, for a given spatial point $x\in \mathbb{R}^3$, outputs its distance to the surface $S$, whose sign encodes whether the point is inside (negative) or outside (positive) of the watertight surface:
\begin{eqnarray*}
    SDF_S(x)=\left\{
    \begin{aligned}
        d(x, S),& \quad\text{ if $x$ is outside},\\
        -d(x, S),& \quad\text{ if $x$ is inside},
    \end{aligned}
    \right.
\end{eqnarray*}
where $d$ is the Euclidean distance. A surface $S$ is implicitly represented by the iso-surface of $SDF_S(\cdot) = 0$ and can be obtained with the marching cube algorithm \cite{lorensen1987marching}.

Modern representation learning techniques aim at automatically discovering a set of features that can compactly but expressively describe the training data. Park et al \cite{park2019deepsdf} proposed a deep neural network (DNN), termed DeepSDF, to directly regress the continuous $SDF(x)$ from point cloud samples, and introduced conditioning on a latent vector to model multiple SDFs with one single neural network. The trained network can predict the $SDF$ value of any query position, from which we can extract the zero level-set surface by evaluating spatial samples. Specifically, the construction proceeds as follows. Given a dataset of surfaces $\mathbb{D}=\{S\}$. For each surface $S\in \mathbb{D}$, a latent vector $z_S\in \mathbb{R}^Z$ is defined in DeepSDF with $Z$ being the dimension of the latent vector. Then one trains a neural network $f_\theta(z,x):\mathbb{R}^Z\times \mathbb{R}^3 \rightarrow \mathbb{R}$ to approximate $SDF(x)$ over the set $\mathbb{D}$ by minimizing the following loss \cite{park2019deepsdf}:
\begin{eqnarray}
    L_{\rm imp}(\{z_S\}_{S \in \mathbb{D}}, \theta)=\sum_{S\in \mathbb{D}}\frac{1}{|X_S|}\sum_{x\in X_S}\|f_\theta(z_S,x)-SDF_S(x)\|_1+\lambda\sum_{S\in \mathbb{D}}\|z_S\|_2^2,\label{deepsdfloss}
\end{eqnarray}
where $\theta$ denotes the vector of trainable DNN parameters, $X_S$ denotes the set of 3D sample points on the surface $S$ (and around it due to the presence of noise), and $\lambda>0$ is the penalty weight. See Figure \ref{fig:deepSDF} for a schematic illustration of DeepSDF architecture.
The choice of the $\ell^1$-norm $\|\cdot\|_1$ of the loss function \eqref{deepsdfloss} makes the learned SDF sharp and and likewise the $\ell^2$-norm makes the label vector smooth.
\begin{figure}[htb]
    \centering
    \includegraphics[scale=0.34, trim = {1cm 6cm 2cm 5cm}, clip]{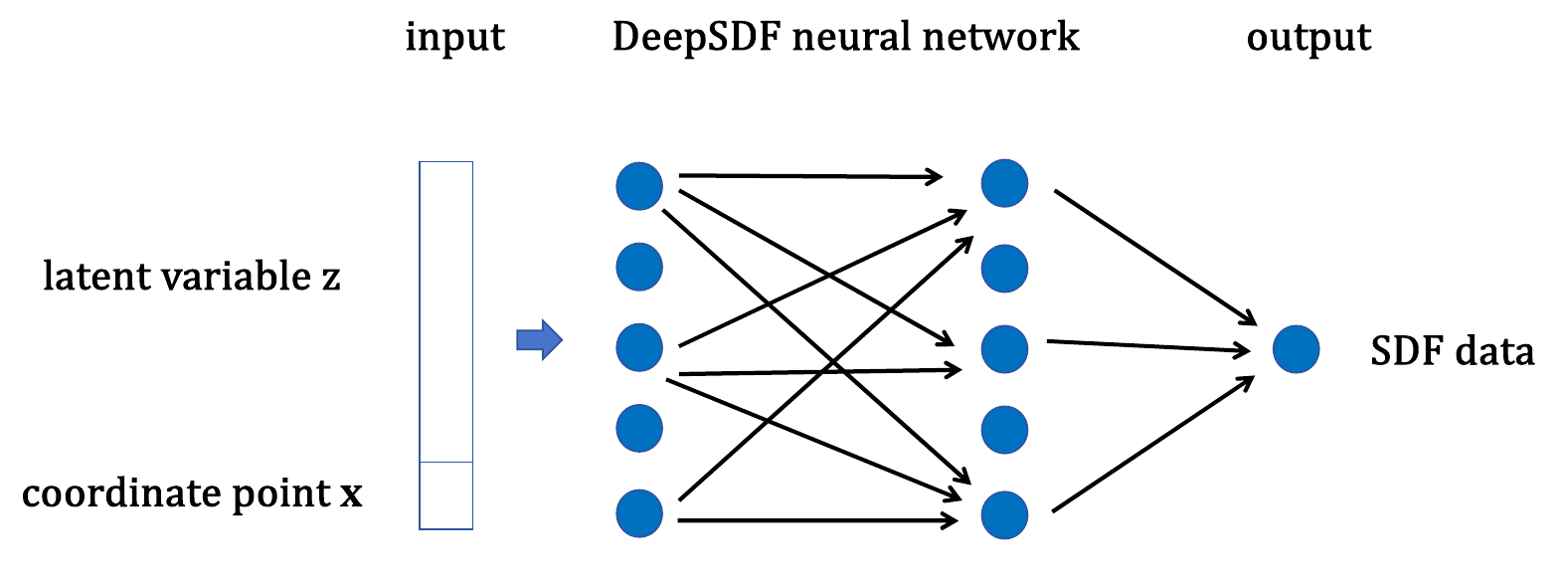}
    \caption{A schematic illustration of DeepSDF \cite{park2019deepsdf}.}	\label{fig:deepSDF}
\end{figure}

Numerically, the DeepSDF representation $f_\theta(z,x)$ is complete and continuous \cite{park2019deepsdf}. The ability of DeepSDF to produce high-quality latent shape manifold opens the door to developing effective reconstruction algorithms for ISP, and DeepSDF is the main tool behind our proposal. One notable advantage of the approach is that the inference can be performed from an arbitrary number of SDF samples. This enables DeepSDF to represent highly complex shapes without discretization errors while with a greatly reduced memory footprint, which differs markedly from the level set type methods that perform updates on a fixed mesh.

Now we can describe the detailed procedure for solving ISP using the latent surface representation. Upon using a trained DeepSDF to represent a surface in ISP, we seek a latent variable $z\in \mathbb{R}^Z$ by minimizing the following cost functional
\begin{eqnarray}\label{Jz}
    \mathcal{L}(z)=J(\Gamma_z)=\frac{1}{2LM} \sum_{l=1}^L \sum_{m=1}^M|u^\infty_{\Omega_z}(\hat{x}_m,d_l)-u^\infty_{\Omega^\star}(\hat{x}_m,d_l)|^2,
\end{eqnarray}
where the domain $\Omega_z$ is given by
\begin{eqnarray*}
    \Omega_z = \{ x\in \mathbb{R}^3| f_\theta(z,x)\le 0\}.
\end{eqnarray*}
The boundary $\Gamma_z$ of the domain $\Omega_z$ is determined implicitly by the equation $f_\theta(z,x)=0$. For a given label $z$, $\Gamma_z$ can be represented by $x=x(z)$ such that $f_\theta(z,x(z))=0$. Below we assume that the implicitly represented surface $\Gamma_z$ is regular \cite{Younes2019}, namely,
$$f_\theta(z,x)=0\Rightarrow \nabla_x f_\theta(z,x)\neq 0.$$
Note that the DeepSDF representation is a deep generative prior approach to ISP. Whenever the label $z$ is different to $\{z_S\}_{S\in \mathbb{D}}$, it will generate a new surface.

We denote the generator function by $G(z)=f_\theta(z,\cdot)$. The function $G$ is assumed to be differentiable in the latent variable $z$, and hence, we can use the standard back-propagation to compute the gradient of the loss involving $G$ in the latent vector $z$. The latter is crucial for employing gradient descent type updates for the training. In practice, each iteration of the method involves two steps that are performed alternatingly: a gradient descent update step and a projection step. The first step applies the shape derivative technique, performing the gradient descent update rule for the loss function \eqref{loss} and gives an approximation $\Gamma_t$, which however may not lie on the manifold. In the second projection step, we minimize the projection loss by the gradient descent update:
\begin{eqnarray*}
	P_G(\Gamma_t)=G(\arg \min_z\|\Gamma_t-G(z)\|),
\end{eqnarray*}
to enforce the manifold constraint. These two steps will be described in detail below.

\subsection{Shape derivative and gradient respect to the latent variable}
First, we recall the shape derivative using the velocity method \cite{sokolowski1992introduction}. Let 
$V:\mathbb{R}\times\mathbb{R}^3\mapsto\mathbb{R}^3$ be a given velocity field. The associated field $x(t,X)$ 
denotes the solution to a system
of ordinary differential equations
\begin{eqnarray}
    \frac{\rm d}{{\rm d} t}x(t,X)=V(t,x(t,X)), \label{velocity}
\end{eqnarray}
for $X\in \Gamma_z$, with the initial condition
$ x(0,X)=X.$
Let $T_t(X)=x(t,X)$. Then $T_t$ defines a transformation from $\Gamma$ to $\Gamma_t(V)=T_t(\Gamma)$. 
The shape derivative $dJ(\Gamma;V)$ of the functional $J(\Gamma)$ in the direction $V$ is defined by
\begin{eqnarray}
    dJ(\Gamma;V)=\lim_{t\rightarrow 0^+}\frac{J(\Gamma_t(V))-J(\Gamma)}{t}.\label{shaped}
\end{eqnarray}
Note that the shape derivative $dJ(\Gamma_z;V)$ only depends on the initial velocity $V(0,X)$ if the functional $J(\Gamma_z)$ is shape differentiable in the Hadamard sense \cite{Delfour2011}.
For more details about shape derivative, we refer to \cite{Delfour2011,sokolowski1992introduction}. 

Let $e_i\in \mathbb{R}^Z$ be the $i$th coordinate vector, with the $i$-th component being 1 and other components zeros. Then we can define a transformation $T(t)$ from the surface $\Gamma_z$ to $\Gamma_{z+te_i}$ as follows. The boundary $\Gamma_z$ is determined implicitly by $f_\theta(z,x)=0$ is $\Gamma_z=\{x\in \mathbb{R}^3|f_\theta(z,x)=0\}$ and $\Gamma_{z+te_i}$ likewise implicitly by $f_\theta(z+te_i,x)=0$ is $\Gamma_{z+te_i}=\{x\in \mathbb{R}^3|f_\theta(z+te_i,x)=0\}$ for $t\geq 0$.
Meanwhile, we may define a vector field $V:\mathbb{R}\times\mathbb{R}^3\mapsto\mathbb{R}^3$ by
\begin{eqnarray}
    V(t,x)=-\frac{\nabla_x f_\theta(z+t e_i,x)}{\|\nabla_x f_\theta(z+t e_i,x) \|^2} \frac{\partial f_\theta}{\partial z_i}(z+t e_i,x).\label{velo}
\end{eqnarray}
Then according to \eqref{velocity}, we derive that for any $X\in \Gamma_z,$
\begin{eqnarray*}
\begin{aligned}
    \frac{\rm d}{{\rm d}t}f_\theta(z+t e_i,x(t,X))=&\frac{\partial f_\theta}{\partial z_i}+\nabla_x f_\theta \cdot \frac{\rm d }{{\rm d} t}x(t,X)\\
    =&\frac{\partial f_\theta}{\partial z_i}-\nabla_x f_\theta \cdot \frac{\nabla_x f_\theta(z+t e_i,X)}{\|\nabla_x f_\theta(z+t e_i,X) \|^2} \frac{\partial f_\theta}{\partial z_i}=0,
\end{aligned}
\end{eqnarray*}
and
\begin{eqnarray*}
    f_\theta(z,x(0,X))=0,\quad \forall X\in \Gamma_z.
\end{eqnarray*}
Consequently, we have
\begin{eqnarray*}
    f_\theta(z+te_i,x(t,X))=0,\quad \forall t\geq 0,\forall X\in \Gamma_z.
\end{eqnarray*}
Let $T_t(X)=x(t,X)$. Then we have $T_t(X)\in \Gamma_{z+te_i}$ for $X\in\Gamma_z$. We assume that the map $T_t: \Gamma_z\rightarrow \Gamma_{z+te_i}$ is surjective for $0\le t\leq \varepsilon$ and $\varepsilon$ is a small constant, then $T_t(\Gamma_z) = \Gamma_{z+te_i}=\{x\in \mathbb{R}^3|f_\theta(z+te_i,x)=0\}$.
Meanwhile, if the functional $J(\Gamma)$ is shape differentiable at $\Gamma_z$ and $V\in C([0,\varepsilon);C^1(\overline{\Omega}_z,\mathbb{R}^3))$, then by \cite[Proposition 2.21]{sokolowski1992introduction}, the following relation holds
\begin{eqnarray}
   dJ(\Gamma_z;V)=\lim_{t\rightarrow 0^+}\frac{J(\Gamma_{z+te_i})-J(\Gamma_z)}{t}=\lim_{t\rightarrow 0^+}\frac{J(T_t(\Gamma_z))-J(\Gamma_z)}{t}.\label{shaped1}
\end{eqnarray}

The derivation of the shape derivative for ISP with the Dirichlet and Neumann boundary conditions was given by Kirsch \cite{Kirsch:1993} and Hettlich \cite{Hettlich:1995}.
Using the identity \eqref{shaped1} and adjoint technique, we have the following representation of the gradient $\nabla \mathcal{L}(z)$ of the loss $\mathcal{L}(z)$ in $z$.
\begin{theorem}\label{Lgradient}
Let $\Omega_z$ be a bounded domain of class $C^2$, and the loss $\mathcal{L}(z)$ be defined in \eqref{Jz}. If the gradient $\nabla_xf_\theta$ of $f_\theta$  does not vanish on  the surface $\Gamma_z$, then there holds
\begin{eqnarray}\label{gradient}
    \nabla \mathcal{L}(z)=-\frac{1}{L} \Re\left(\sum_{l=1}^L\int_{\Gamma_z}{\frac{\partial u_l}{\partial \nu}\frac{\partial w_l}{\partial \nu} \frac{\nabla_z f_\theta}{\|\nabla_x f_\theta\|} {\rm d}s}\right),
\end{eqnarray}
where $\Re$ denotes taking the real part of a complex number, $\nu$ is the unit outward normal vector to the boundary $\Gamma_z$, $u_l$ is the total field generated by the incident plane wave with the incident direction $d_l$ of problem \eqref{hel} and $w_l$ is the solution of the respective adjoint problem defined in \eqref{adjoint} below.
\end{theorem}
\begin{proof}
By the definition of $\mathcal{L}(z)$, we have
\begin{eqnarray*}
    \frac{\partial\mathcal{L}(z)}{\partial z_i}=\lim_{t\rightarrow 0^+}\frac{J(\Gamma_{z+te_i})-J(\Gamma_z)}{t}.
\end{eqnarray*}
By the definition \eqref{shaped} of the shape derivative,  we can derive from \eqref{shaped1} that
\begin{eqnarray}\label{gradient1}
   \lim_{t\rightarrow 0^+}\frac{J(\Gamma_{z+te_i})-J(\Gamma_z)}{t}=dJ(\Gamma_z;V),
\end{eqnarray}
with $V$ defined by \eqref{velo}.
Since $\Omega_z$ is a bounded domain of class $C^2$, it follows from \cite[Theorem 3.2]{audibert2022accelerated} that  the shape derivative of $J(\Gamma)$ at $\Gamma_z$ is given by
\begin{eqnarray}\label{shape}
    dJ(\Gamma_z;V)=\frac{1}{L} \Re\left(\sum_{l=1}^L\int_{\Gamma_z}{\frac{\partial u_l}{\partial \nu}\frac{\partial w_l}{\partial \nu}(\nu\cdot V(0)){\rm d}s}\right), 
\end{eqnarray}
where $u_l$ is the total field of problem \eqref{hel} (with the incident direction $d = d_l$) and $w_l := w_l^i + w_l^s$ is the solution to the following adjoint problem
\begin{eqnarray}\label{adjoint}
\left\{	\begin{aligned}
            \Delta w_l + k^2 w_l  &=0,  \quad \mbox{in }\mathbb{R}^3 \backslash \Omega_z,   \\
            w_l &= 0, \quad\mbox{on } \Gamma_z, \\
            \frac{\partial w_l^s}{\partial r}-ikw_l^s & =o(r^{-1}),\quad \mbox{as }r\to \infty,
    \end{aligned}
\right.\end{eqnarray}
with the incident wave
\begin{eqnarray*}
    w_l^i(y)=\frac{1}{4\pi M}\sum_{m=1}^{M}\overline{ (u^\infty_{\Omega_z}(\hat{x}_m,d_l)-u^\infty_{\Omega^*}(\hat{x}_m,d_l))}e^{-ik\hat{x}_m \cdot y},
\end{eqnarray*}
where the overbar refers to complex conjugate.
Plugging the expression
\begin{eqnarray*}
    V(0)\cdot \nu=-\frac{\nabla_x f_\theta}{\|\nabla_x f_\theta \|^2} \frac{\partial f_\theta}{\partial z_i}\cdot \frac{\nabla_x f_\theta}{\|\nabla_x f_\theta\|}=-\frac{1}{\|\nabla_x f_\theta\|}\frac{\partial f_\theta}{\partial z_i}
\end{eqnarray*}
into \eqref{shape} yields the desired result.
\end{proof}

Note that the latent variable $z$ is the input of the DNN $f_\theta$, the gradients $\nabla_x f_\theta$ and $\nabla_z f_\theta$ can both be computed via automatic differentiation with backpropagation \cite{Goodfellow2016}, e.g.,  \texttt{torch.autograd} in PyTorch. 

Once we have the gradient $\nabla \mathcal{L}(z)$, we can use the ADAM algorithm \cite{KingmaBa:2015} to update the latent variable $z$. Specifically, at each iteration, by randomly choosing partial components of $\nabla \mathcal{L}(z_n)$, we get a stochastic gradient $g_n$ such that $\mathbb{E} [g_n]=\nabla \mathcal{L}(z_n)$. This can be implemented by, e.g., \texttt{torch.optim.Adam} in PyTorch. Following standard practice, we take $0 < \beta_2 \le 1, 0 \le \beta_1 < \beta_2,$ and a nonnegative sequence $(\alpha_n)_{n\in \mathbb{N}}$, and define three vectors $m_n,v_n,z_n\in \mathbb{R}^Z$ iteratively. More precisely, given the initial guess $z_0\in \mathbb{R}^Z$, $m_0 = 0$, and $v_0 = 0$, we define for all indices $n\in \mathbb{N}$,
\begin{align}\label{adam}
\left\{\begin{aligned}
    m_{n,i}=&\beta_1 m_{n-1,i}+(1-\beta_1)g_{n-1,i},\\
    v_{n,i}=&\beta_2 v_{n-1,i}+(1-\beta_2)(g_{n-1,i})^2,\\
    z_{n,i}=&z_{n-1,i}-\alpha_n \frac{m_{n,i}}{\sqrt{v_{n,i}}},
\end{aligned}\right.
\end{align}
where $x_j$ denotes $j$-th coordinate of a vector $x$, and $\alpha_n$ is the learning rate at step $n$, $\beta_{1}>0$ is a momentum parameter, and $\beta_2$ controls the decay rate of the per-coordinate exponential moving average of the squared gradient. The procedure is given in Algorithm \ref{algorithm}.

\begin{algorithm}[t]
\caption{Iterative recovery with latent represent surfaces}
\label{algorithm}
\begin{algorithmic}[1]
    \State {\bf Input:} Far field pattern, wave number $k$, directions $\{d_l\}_{l=1}^L$, $\{\hat x_m\}_{m=1}^M$.
    \State  {\bf Initialization:} initial label $z_0$, step-size sequence:$\{\alpha_n\}_{n=0}^{N}$, $m_{0}=0$ and $v_{0}=0$.
    \For{$n=0,1,2,...N$}
    \State Generate the surface $\Gamma$ for $z_n$ (with marching cubes algorithm).
    \State Compute $u_l^n$ on $\Gamma$ by solving problem \eqref{hel}. 
    \State Compute  $w_l^n$ on $\Gamma$ by solving problem \eqref{adjoint}.
    \State Compute the gradient $\nabla \mathcal{L}(z_{n})$ with \eqref{gradient}.
    \State Generate an unbiased sample $g_n$ (i.e., $\mathbb{E} [g_n]=\nabla \mathcal{L}(z_n)$).
    \State $m_{n+1,i}=\beta_1 m_{n,i}+(1-\beta_1)g_{n-1,i}$.
    \State $v_{n+1,i}=\beta_2 v_{n,i}+(1-\beta_2)(g_{n-1,i})^2$.
    \State $z_{n+1,i}=z_{n,i}-\alpha_{n+1} \frac{m_{n+1,i}}{\sqrt{v_{n+1,i}}}$.
    \EndFor\\
    \Return Surface with latent $z_{N+1}$.
\end{algorithmic}
\end{algorithm}

\section{Convergence of the optimization algorithm}\label{analysis}
Now we discuss the convergence of the scheme \eqref{adam}, which provides preliminary mathematical guarantees for the imaging algorithm. First, we define the second-order shape derivative. Let $V_1,V_2:\mathbb{R}\times\mathbb{R}^3\mapsto\mathbb{R}^3$ be two differentiable velocity fields. Let the objective $J(\Gamma)$ be shape differentiable in the sense of Hadamard \cite{Delfour2011}. Then we define the second derivative by
\begin{eqnarray*}
    d^2J(\Gamma;V_1;V_2)=\lim_{t\rightarrow 0^+}\frac{dJ(\Gamma_t(V_1);V_2(t))-dJ(\Gamma;V_2(0))}{t}.
\end{eqnarray*}
It follows from   \cite[Theorem 6.2, Chapter 9]{Delfour2011} that
\begin{eqnarray}
    d^2J(\Gamma;V_1;V_2)=d^2J(\Gamma;V_1(0);V_2(0))+dJ(\Gamma;V_1'(0)),\label{J2OVV}
\end{eqnarray}
with
\begin{align*}
    d^2J(\Gamma;V_1(0);V_2(0))&:=\lim_{t\rightarrow 0^+}\frac{dJ(T_t(\Gamma,V_1);V_2(0))-dJ(\Gamma;V_2(0))}{t},\\
    dJ(\Gamma;V(0))&:= \lim_{t\rightarrow 0^+}\frac{J((T_t(\Gamma,V))-J(\Gamma)}{t},\\
    T_t(\Gamma,V)&=\{x+tV(0,x)|x\in\Gamma\},
\end{align*}
and
\begin{eqnarray*}   
    V'(0)=\lim_{t\rightarrow0+}\frac{V(t,x)-V(0,x)}{t}.
\end{eqnarray*}
A detailed discussion of second-order shape derivatives can be found in the monographs \cite{Delfour2011,simon1989second} and \cite[appendix]{hettlich1999second}. First, we provide useful a priori bounds on the shape derivatives.
\begin{theorem}\label{hessian}
Let the boundary $\Gamma$ be of  class $C^5$. Then there is $C$ independent of $V$ such that
\begin{align*}
    |dJ(\Gamma;V(0))|&\le C\|V(0)\|_{C(\Gamma;\mathbb{R}^3)}, \quad\forall  V\in C([0,\varepsilon);C(\overline{\Omega},\mathbb{R}^3)),\\ 
    |d^2J(\Gamma;V_1(0);V_2(0))|&\le C\big\|V_{1}(0)\|_{C^{3}(\Gamma,\mathbb{R}^3)}\|V_{2}(0)\|_{C^{3}(\Gamma,\mathbb{R}^3)},\,\quad \forall V_1,V_2\in C([0,\varepsilon);C^3(\overline{\Omega},\mathbb{R}^3)),
\end{align*}
where $\varepsilon>0$ is a small constant, $\|f\|_{C^m(\Gamma)}\triangleq \sum_{i=0}^m\|f^{(i)}\|_{C(\Gamma)}$ and $\|f\|_{C(\Gamma)}\triangleq \max_{x\in \Gamma}|f(x)|$ for any $f(x)\in C^m(\overline{\Omega}, \mathbb{R}^3)$.
\end{theorem}

The proof of the theorem requires several preliminary results on fractional order Sobolev spaces in the appendix, and a classical regularity result for elliptic PDEs \cite[Theorem 2.6.7]{nedelec2001acoustic}. 
\begin{lemma}\label{regular}
If the boundary $\Gamma$ is of  class  $C^{m}$, $u^i \in H^{m-1/2}(\Gamma),\, m\in \mathbb{N}$, then the solution $u$ of problem \eqref{hel} satisfies $u\in H^{m}(\Omega^+)$, and there exists $C=C(k,\Omega)$ such that
\begin{eqnarray*}
    \|u\|_{H^{m}(\Omega^+)}\le C \|u^i\|_{H^{m-1/2}(\Gamma)},
\end{eqnarray*}
with $\Omega^+:=\mathbb{R}^3 \backslash \Omega$,
$H(\Omega^+)=\{u, \|u\|_{H(\Omega^+)}:=\int_{\Omega^+} \frac{u^2}{r^2+1}+\frac{|\nabla u|^2}{r^2+1}{\rm d}V+\int_{\Omega^+} |\frac{\partial u}{\partial r}-iku|^2 {\rm d}V<\infty\}$, and
$H^m(\Omega^+)=\{u, \|u\|_{H^m(\Omega^+)}:=\sum_{|\alpha|\le m}\|D^\alpha u\|_{H(\Omega^+)}<\infty\}$, with $ r=\|x\|_2$.
\end{lemma}

\begin{lemma}\label{DtN}
Let $u$ be the solution of problem \eqref{hel} and the boundary $\Gamma$ be of class $C^m$, $m\in \mathbb{N}$ and $m\geq 1$. Then $\gamma u \in H^{m-1/2}(\Gamma), \partial_\nu u \in H^{m-3/2}(\Gamma)$, and
\begin{eqnarray*}
    \|\partial_{\nu}u\|_{H^{m-3/2}(\Gamma)}\le C\|\gamma u\|_{H^{m-1/2}(\Gamma)}.
\end{eqnarray*}
\end{lemma}
\begin{proof}
By Lemma \ref{regular}, if $u^i\in H^{m-1/2}(\Gamma)$, then the solution $u\in H^m(\Omega^+)$ and there holds
$$\|u\|_{H^m(\Omega^+)}\leq C\|u^i\|_{H^{m-1/2}(\Gamma)}=C\|\gamma u\|_{H^{m-1/2}(\Gamma)}.$$
By the trace theorem \cite[pp. 163--164]{AdamsFournier:2003},  we deduce
$\|\partial_{\nu} u\|_{H^{m-3/2}(\Gamma)}\le C\|u\|_{H^m(\Omega^+)}$.
Combining these two estimates yields the desired assertion.
\end{proof}

To prove Theorem \ref{hessian}, we define the far field pattern operator $F: \mathcal{T} \rightarrow L^2(\mathbb{S}^2)$, which maps the admissible boundary $\Gamma$ from the set $\mathcal{T}$ to the far field patterns of problem \eqref{hel} for a fixed incident direction $d \in \mathbb{S}^2$ (not explicitly indicated below) and a measurement direction $\hat{x} \in \mathbb{S}^2$, defined by
$F(\Gamma) = u_\infty$ and $u_\infty$ is the far field pattern defined by \eqref{ffield}. By \cite[Theorem 6.1]{hettlich1999second}, the operator $F$ is twice continuously differentiable and 
\begin{equation*}
    d^2F(\Gamma;V_1(0);V_2(0))=u''_\infty,
\end{equation*}
for variations $V_1,V_2 \in C([0,\varepsilon);C^1(\overline{\Omega},\mathbb{R}^3))$ of the boundary $\Gamma$. The function $u''_\infty$ is the far field pattern of the radiating solution of the scattering problem
\begin{equation}\label{secondorder}
\left\{ \begin{aligned}
    \Delta u'' + k^2 u'' & =0,  \quad \mbox{in }\mathbb{R}^3 \backslash \Omega,\\
    u''& = \psi, \quad\mbox{on } \Gamma,\\
    \frac{\partial u''_s}{\partial r}-iku''_s &=o(r^{-1}),\quad \mbox{as }r\to \infty,
\end{aligned}\right.
\end{equation}
with the boundary data $\psi$ given by
\begin{align}
    \psi=&-V_{1, \nu}\frac{\partial u'_2}{\partial \nu}-V_{2, \nu}\frac{\partial u'_1}{\partial \nu}+(V_{1, \nu}V_{2, \nu}-V_{1, \tau}V_{2, \tau})\kappa \frac{\partial u}{\partial \nu}\nonumber\\
      &+(V_{1, \tau}(\tau \cdot \nabla V_{2, \tau})+V_{2, \tau}(\tau \cdot \nabla V_{1, \tau}))\frac{\partial u}{\partial \nu},\label{eqn:u''-boundary}
\end{align} 
where $u$ is the solution of the scattering problem \eqref{hel}, and $\nu$ denotes the unit outward normal to the boundary $\Gamma$. In the expression \eqref{eqn:u''-boundary}, the notation $V_{j,\nu}=V_j(0)\cdot \nu$ and $V_{j,\tau} = V_j(0) \cdot \tau$ denote the normal and tangential components of the vector $V_j(0)$, respectively, $\tau \cdot \nabla$ denotes taking the tangential gradient, and $\kappa$ denotes the curvature of the boundary $\Gamma$. Moreover, $u'_j$, for $j=1,2$, is the solution to the following boundary value problem
\begin{align}\label{firstorder0}
\left\{\begin{aligned}
    \Delta u' + k^2 u'&=0,  \quad \mbox{in }\mathbb{R}^3 \backslash \Omega,\\
    \quad u'&=-V_{j,\nu}\frac{\partial u}{\partial \nu}, \quad\mbox{on } \Gamma,\\
    \quad \frac{\partial u'}{\partial r}-iku'&=o(r^{-1}), \quad \mbox{as } r\to \infty.
\end{aligned}\right.
\end{align}
Note that the function $u_j'$ is associated with the vector field $V_j$ in the Dirichlet boundary condition. Still by \cite[p.~25]{hettlich1999second}, with the help of problem \eqref{firstorder0}, we have
\begin{equation*}
    dF(\Gamma;V(0))=u'_{\infty},
\end{equation*}
where $u'_{\infty}$ is the far field pattern of the solution $u'$ to problem \eqref{firstorder0} with the Dirichlet boundary condition $u'=-V(0)\cdot \nu \frac{\partial u}{\partial \nu}$ on the boundary $\Gamma$. For the operator $F$, we have the following useful estimates.
\begin{lemma}\label{derivativeF}
Under the assumption of Theorem \ref{hessian}, the following estimates hold
\begin{align*}
    |dF(\Gamma;V(0))|&\le C\|V(0)\|_{C^1(\Gamma;\mathbb{R}^3)}, \quad\forall  V\in C([0,\varepsilon);C^1(\overline{\Omega},\mathbb{R}^3)),\\ 
    |d^2F(\Gamma;V_1(0);V_2(0))|&\le C\|V_{1}(0)\|_{C^{3}(\Gamma,\mathbb{R}^3)}\|V_{2}(0)\|_{C^{3}(\Gamma,\mathbb{R}^3)},\,\quad \forall V_1,V_2\in C([0,\varepsilon);C^3(\overline{\Omega},\mathbb{R}^3)).
\end{align*}
\end{lemma}
\begin{proof}
First, the $C^5$ regularity of the boundary $\Gamma$ and Lemma \ref{DtN} imply that for any bounded subdomain $\Omega^+=\mathbb{R}^3 \backslash \overline{\Omega}$,
\begin{equation}\label{eqn:u-apriori}
    u\in H^5(\Omega^+),\quad \gamma u \in H^{9/2}(\Gamma) \quad \mbox{and}\quad \partial_\nu u \in H^{7/2}(\Gamma).
\end{equation}
By the definition \eqref{ffield} of the far-field pattern operator $F(\Gamma)$ and the estimates 
$\|\partial_\nu e^{-ik\hat{x}\cdot y}\|_{H^{-1/2}(\Gamma)}\leq C$ and $\|e^{-ik\hat{x}\cdot y}\|_{H^{1/2}(\Gamma)}\leq C$, we have
\begin{align*}
    |dF(\Gamma;V(0))|=&\left|\frac{1}{4\pi}\int_{\Gamma} \frac{\partial e^{-ik\hat{x}\cdot y}}{\partial \nu}u'(y) - \frac{\partial u'(y)}{\partial \nu}e^{-ik\hat{x}\cdot y} {\rm d}s(y)\right|\\
    \le & \frac{1}{4\pi}\left(\left|\int_{\Gamma} \frac{\partial e^{-ik\hat{x}\cdot y}}{\partial \nu}u'(y)ds(y)\right| + \left|\int_{\Gamma}\frac{\partial u'(y)}{\partial \nu}e^{-ik\hat{x}\cdot y} {\rm d}s(y)\right|\right)\\
    \le&C(\|\partial_\nu e^{-ik\hat{x}\cdot y}\|_{H^{-1/2}(\Gamma)}\|u'\|_{H^{1/2}(\Gamma)}+\|\partial_{\nu}u'\|_{H^{-1/2}(\Gamma)}\|e^{-ik\hat{x}\cdot y}\|_{H^{1/2}(\Gamma)})\\
    \le& C(\|u'\|_{H^{1/2}(\Gamma)}+\|\partial_{\nu}u'\|_{H^{-1/2}(\Gamma)}).
\end{align*}
The boundedness of the DtN map for the Helmholtz equation (cf. Lemma \ref{DtN}) implies
\begin{align*}
    &|dF(\Gamma;V(0))|\le C\|\gamma u'\|_{H^{1/2}(\Gamma)}=C\big\|-V_{\nu}\partial_{\nu}u \big\|_{H^{1/2}(\Gamma)}.
\end{align*}
Then by Lemma \ref{fg}, we have
\begin{align*}
     |dF(\Gamma;V(0))| \le C\big(\|V_{\nu}\|_{C(\Gamma)}\|\partial_{\nu}u\|_{H^{1/2}(\Gamma)}+\|V_{\nu}\|_{H^{1/2}(\Gamma)}\|\partial_{\nu}u\|_{C(\Gamma)}+\|V_{\nu}\|_{C(\Gamma)}\|\partial_{\nu}u\|_{C(\Gamma)}\big).
\end{align*}
Now the continuous Sobolev embedding $H^{s}(\Gamma)\hookrightarrow C^{r,\alpha-\epsilon}(\Gamma)$ for $s>1$ and $s-1=r+\alpha$ with $\alpha \in (0, 1)$ and Lemma \ref{fraction} imply
\begin{align*}
    |dF(\Gamma;V(0))| & \le C\|V_{\nu}\|_{C(\Gamma)}\|\partial_{\nu} u\|_{H^{1/2}(\Gamma)}+C\|\partial_{\nu}u\|_{H^{3/2}(\Gamma)}\big(\|V_{\nu}\|_{H^{1/2}(\Gamma)}+\|V_{\nu}\|_{C(\Gamma)}\big)\\
    &\le C\|\partial_{\nu} u\|_{H^{3/2}(\Gamma)}\big(\|V_{\nu}\|_{H^{1/2}(\Gamma)}+\|V_{\nu}\|_{C(\Gamma)}\big).
\end{align*}
This, the estimate $\|\partial_\nu u\|_{H^{7/2}(\Gamma)}\leq C$ from \eqref{eqn:u-apriori} and Lemma \ref{fraction} yield 
\begin{align}
    |dF(\Gamma;V(0))|\le&C\big(\|V_{\nu}\|_{H^{1/2}(\Gamma)}+\|V_{\nu}\|_{C(\Gamma)}\big)\le C\|V_{\nu}\|_{C^{1}(\Gamma)}\le C\|V(0)\|_{C^{1}(\Gamma)}.\label{firstDF}
\end{align}
This shows the first assertion.
Next,  by Lemma \ref{DtN}, we have
\begin{align*}
 &|d^2F(\Gamma;V_1(0);V_2(0))|=\left|\frac{1}{4\pi}\int_{\Gamma} \frac{\partial e^{-ik\hat{x}\cdot y}}{\partial \nu}u''(y) - \frac{\partial u''(y)}{\partial \nu}e^{-ik\hat{x}\cdot y} {\rm d}s(y)\right|\\
    \le&C(\|\partial_\nu e^{-ik\hat{x}\cdot y}\|_{H^{-1/2}(\Gamma)}\|\gamma u''\|_{H^{1/2}(\Gamma)}+\|\partial_{\nu}u''\|_{H^{-1/2}(\Gamma)}\|e^{-ik\hat{x}\cdot y}\|_{H^{1/2}(\Gamma)})\\
    \le& C(\|\gamma u''\|_{H^{1/2}(\Gamma)}+\|\partial_{\nu}u''\|_{H^{-1/2}(\Gamma)})
    \le C\|\gamma u''\|_{H^{1/2}(\Gamma)}
\end{align*}
From the expression of the boundary data $\psi$ in \eqref{eqn:u''-boundary} and the triangle inequality, we deduce
\begin{align*}
    &|d^2F(\Gamma;V_1(0);V_2(0))|\le C\|\gamma u''\|_{H^{1/2}(\Gamma)} =C\big\|\psi\big\|_{H^{1/2}(\Gamma)}\\
     \leq& C\Big(\big\|V_{1, \nu}\partial_{\nu}u'_2\big\|_{H^{1/2}(\Gamma)}+\big\|V_{2, \nu}\partial_{\nu}u'_1\big\|_{H^{1/2}(\Gamma)}+\big\|(V_{1, \nu}V_{2, \nu}-V_{1, \tau}V_{2, \tau})\kappa \partial_{\nu}u\big\|_{H^{1/2}(\Gamma)}\\ 
    & +\big\|V_{1, \tau}(\tau \cdot \nabla V_{2, \tau})\partial_{\nu}u \big\|_{H^{1/2}(\Gamma)}
    +\big\|V_{2, \tau}(\tau \cdot \nabla V_{1, \tau})\partial_{\nu}u\big\|_{H^{1/2}(\Gamma)}\Big):=\sum_{i=1}^5 {\rm I}_i.
\end{align*}
We bound the five terms separately. For the first term $\rm I_1$, by Lemma \ref{fg}, we have
\begin{equation*}
    {\rm I}_1  \le C\big(\|V_{1, \nu}\|_{C(\Gamma)}\|\partial_{\nu}u'_2\|_{H^{1/2}(\Gamma)}+\|\partial_{\nu}u'_2\|_{C(\Gamma)}\|V_{1, \nu}\|_{H^{1/2}(\Gamma)}+\|\partial_{\nu}u'_2\|_{C(\Gamma)}\|V_{1, \nu}\|_{C(\Gamma)}\big).
\end{equation*}
By Sobolev embedding $H^{3/2}(\Gamma)\hookrightarrow C(\Gamma)$ and Lemma \ref{fraction}, we deduce
\begin{align*}
    {\rm I}_1 
    &\le C\|\partial_{\nu} u_2'\|_{H^{3/2}(\Gamma)}\big(\|V_{1, \nu}\|_{H^{1/2}(\Gamma)}+\|V_{1, \nu}\|_{C(\Gamma)}\big).
\end{align*}
Meanwhile, by the boundedness of the DtN map (cf. Lemma \ref{DtN}), we have
\begin{align*}
    &\|\partial_{\nu} u_2'\|_{H^{3/2}(\Gamma)} \le C\|\gamma u_2'\|_{H^{5/2}(\Gamma)}=C\|V_{2,\nu}\partial_{\nu}u\|_{H^{5/2}(\Gamma)}\\
    \le&C\big(\|V_{2,\nu}\|_{C^{2}(\Gamma)}\|\partial_{\nu}u\|_{H^{5/2}(\Gamma)}+\|\partial_{\nu}u\|_{C^{2}(\Gamma)}\|V_{2,\nu}\|_{H^{5/2}(\Gamma)}
    +\|\partial_{\nu}u\|_{C^{2}(\Gamma)}\|V_{2,\nu}\|_{C^{2}(\Gamma)}\big).
\end{align*}
The regularity estimate \eqref{eqn:u-apriori}, Sobolev embedding $H^{7/2}(\Gamma)\hookrightarrow C^2(\Gamma)$, and Lemma \ref{fraction} imply
\begin{align*}
    {\rm I}_1\le&C\big(\|V_{2,\nu}\|_{H^{5/2}(\Gamma)}+\|V_{2,\nu}\|_{C^{2}(\Gamma)}\big)\big(\|V_{1,\nu}\|_{H^{1/2}(\Gamma)}+\|V_{1,\nu}\|_{C(\Gamma)}\big)\le C\|V_{2,\nu}\|_{C^{3}(\Gamma)}\|V_{1,\nu}\|_{C^{1}(\Gamma)}.
\end{align*}
The proof also gives the estimate
$ {\rm I}_2\le C\|V_{1,\nu}\|_{C^{3}(\Gamma)}\|V_{2,\nu}\|_{C^{1}(\Gamma)}$.
Next from Lemmas \ref{fg} and \ref{fraction} and the regularity estimate \eqref{eqn:u-apriori}, we deduce
\begin{align*}
    {\rm I}_3  
    \le C\|V_1(0)\|_{C^1(\Gamma)}\|V_2(0)\|_{C^1(\Gamma)}.
\end{align*}
Last, using Lemmas \ref{fg} and \ref{fraction} and \eqref{eqn:u-apriori} again yields
\begin{align*}
    {\rm I}_4 \le  C \|V_1(0)\|_{C^1(\Gamma)}\|V_2(0)\|_{C^{2}(\Gamma)}\quad \mbox{and}\quad
    {\rm I}_5  \le C \|V_1(0)\|_{C^{2}(\Gamma)}\|V_2(0)\|_{C^1(\Gamma)}.
\end{align*}
Combining these estimates yields the second assertion, completing the proof.
\end{proof}

Now we are in the position to prove Theorem \ref{hessian}.
\begin{proof}[Proof of Theorem \ref{hessian}]
It follows from the representation \eqref{shape} of the shape derivative 
and Lemmas \ref{regular} and \ref{DtN} that for any $V \in C([0,\varepsilon);C^1(\overline{\Omega},\mathbb{R}^3))$,
\begin{align*}
    &|dJ(\Gamma;V(0))|
    \le \Big(\frac{1}{L}\sum_{l=1}^L\Big\|\frac{\partial u_l}{\partial \nu}\Big\|_{L^2(\Gamma)}\Big\|\frac{\partial w_l}{\partial \nu}\Big\|_{L^2(\Gamma)}\Big)\|V(0)\cdot \nu\|_{C(\Gamma)}\\
    \le& \Big(\frac{1}{L}\sum_{l=1}^L\|u_l\|_{H^1(\Gamma)}\|w_l\|_{H^1(\Gamma)}\Big)\|V(0)\cdot \nu\|_{C(\Gamma)} \le C\|V(0)\|_{C(\Gamma;\mathbb{R}^3)}.
\end{align*}
Similarly, it follows from the identity
\begin{equation*}
    d^2J(\Gamma;V_1(0);V_2(0))=\sum_{l,m}dF_{l,m}(\Gamma;V_1(0))dF_{l,m}(\Gamma;V_2(0))+
    (F_{l,m}(\Gamma)-u_{l,m}^*)d^2F_{l,m}(\Gamma;V_1(0);V_2(0)),
\end{equation*}
and Lemma \ref{derivativeF} (for any $d_l \in S^1$ and $\hat{x}_m \in S^1$) that
\begin{align*}
    &|d^2J(\Gamma;V_1(0);V_2(0))| \\
    \le&\sum_{l,m}|dF_{l,m}(\Gamma;V_1(0))dF_{l,m}(\Gamma;V_2(0))+(F_{l,m}(\Gamma)-u^*_{l,m})d^2F_{l,m}(\Gamma;V_1(0);V_2(0))|\\
    \le& C\|V_{1}(0)\|_{C^1(\Gamma)}\|V_{2}(0)\|_{C^1(\Gamma)}+C\|V_{1}(0)\|_{C^{3}(\Gamma)}\|V_{2}(0)\|_{C^{3}(\Gamma)}\\
    \le& C\|V_{1}(0)\|_{C^{3}(\Gamma)}\|V_{2}(0)\|_{C^{3}(\Gamma)}.
\end{align*}
This completes the proof of the theorem.
\end{proof}

Now, we can give the convergence of the Adam algorithm. In the proof, we use the notation $[N]$ to denote the set $\{1,2,,,N\}$, and use $\mathcal{O}(\cdot)$ as standard asymptotic notation. 
\begin{theorem}
Suppose that the sequence $\{z_n\}_{n=0}^\infty\subset \mathbb{R}^Z$ generated by Algorithm \ref{algorithm} is uniformly bounded, and the following two assumptions hold: {\rm(i)} The algorithm can access a bounded noisy gradient $g_n$ and the true gradient $\nabla \mathcal{L}(z_n)$ is bounded, i.e. $\|\nabla \mathcal{L}(z_n)\| \le H$, $\|g_n\| \le H, \forall n > 1$; {\rm (ii)} The noisy gradient $g_n$ is unbiased and the noise is independent, i.e., $g_n = \nabla \mathcal{L}(z_n) + \zeta_n, E[\zeta_n] = 0$ and $\zeta_i$ is independent of $\zeta_j$ if $i \neq j$.
Let the assumptions of Theorem \ref{hessian} be fulfilled with hyper-parameters $\beta_1 \in [0, 1)$ non-increasing, and for some constant $G >0$, $\|\alpha_n m_n/\sqrt{\hat{v}_n}\|\le G$, for all $ n$. Then the following estimate holds
\begin{align}\label{bounded}
\begin{aligned}
    &\mathbb{E}\left[\sum_{n=1}^N \alpha_n\Big\langle\nabla \mathcal{L}    (z_n),\frac{\nabla \mathcal{L}(z_n)}{\sqrt{\hat{v}_n}}\Big\rangle\right]\\
    \le&\mathbb{E}\left[C_1 \sum_{n=1}^N\Big\|\frac{\alpha_n g_n}{\sqrt{\hat{v}_n}}\Big\|^2+C_2\sum_{n=2}^N\Big\|\frac{\alpha_n}{\sqrt{\hat{v}_n}}-\frac{\alpha_{n-1}}{\sqrt{\hat{v}_{n-1}}}\Big\|_1+C_3\sum_{n=2}^{N-1}\Big\|\frac{\alpha_n}{\sqrt{\hat{v}_n}}-\frac{\alpha_{n-1}}{\sqrt{\hat{v}_{n-1}}}\Big\|^2\right]+C_4,
\end{aligned}
\end{align}
where $C_1$, $C_2$, $C_3$ and $C_4$ are constants independent of $N$, the expectation $\mathbb{E}$ is taken with respect to the randomness of $\{g_n\}$.
Further, let $\gamma_n := \min_{j\in[Z]} \min_{\{g_i\}^n_{i=1}} \alpha_n m_n/(\sqrt{\hat{v}_n})_j$ denote the minimum possible value of effective stepsize at time $n$ over all possible coordinate and past gradients $\{g_i\}_{i=1}^{n}$. Then the convergence rate of Algorithm \ref{algorithm} is given by
\begin{eqnarray}\label{rate}
    \min_{i\in[N]} \mathbb{E}[\|\nabla \mathcal{L}(z_n)\|^2]=\mathcal{O}\Big(\frac{s_1(N)}{s_2(N)}\Big),
\end{eqnarray}
where $s_1(N)$ is the upper bound in \eqref{bounded}, and $s_2(N)=\sum_{n=1}^N \gamma_n$.
\end{theorem}
\begin{proof}
Since the loss $\mathcal{L}(z)$ is nonnegative, it is bounded below by zero, and there exists an infimum to $\mathcal{L}(z)$. If $\mathcal{L}$ is differentiable and has an $L$-Lipschitz gradient, i.e., 
\begin{equation*}
\|\nabla \mathcal{L}(z)-\nabla \mathcal{L}(\tilde z) \| \le L\|z - \tilde z\|,\quad \forall z,\tilde z,
\end{equation*}
the desired assertions \eqref{bounded} and \eqref{rate} are direct from  \cite[Theorem 3.1]{chen2019convergence}. So it suffices to show the $L$-Lipschitz continuity of the loss $\mathcal{L}(z)$. First, we define the three vector fields $V_1$, $V_2$, and $W$ by
\begin{align*}
    V_1[z](t,x)&=-\frac{\nabla_x f_\theta(z+t e_i,x)}{\|\nabla_x f_\theta(z+t e_i,x) \|^2} \frac{\partial f_\theta}{\partial z_i}(z+t e_i,x),\\
    V_2[z](t,x)&=-\frac{\nabla_x f_\theta(z+t e_j,x)}{\|\nabla_x f_\theta(z+t e_j,x) \|^2} \frac{\partial f_\theta}{\partial z_j}(z+t e_j,x),\\
    W(t,x)&=-\frac{\nabla_x f_\theta(z+te_j,x)}{\|\nabla_x f_\theta(z+te_j,x) \|^2} \frac{\partial f_\theta}{\partial z_i}(z+te_j,x).
\end{align*}
Then we have $W(t,x)=V_1[z+te_j](0,x)$ with $W(0,x)=V_1(0,x)$.
It follows from \eqref{gradient1} that $\frac{\partial \mathcal{L}}{\partial {z_i}}(z)=dJ(\Gamma_z;V_1[z])$, and 
\begin{align*}
    \frac{\partial^2 \mathcal{L}(z)}{\partial {z_i} \partial {z_j}}=&\frac{\partial }{\partial {z_j}}dJ(\Gamma_z;V_1[z])
    =\lim_{t\rightarrow 0^+}\frac{dJ(\Gamma_{z+te_j};V_1[z+te_j](0))-dJ(\Gamma_z;V_1)}{t}\\
    =&\lim_{t\rightarrow 0^+}\frac{dJ(\Gamma_{t}(V_2);W(t))-dJ(\Gamma_z;V_1)}{t}
    =d^2J(\Gamma_z;W;V_2).
\end{align*}
Then by \eqref{J2OVV}, we have
\begin{align*}
    \frac{\partial^2 \mathcal{L}(z)}{\partial {z_i} \partial {z_j}}
    =&d^2J(\Gamma_z;W(0);V_2(0))+d^2J(\Gamma;W'(0))\\
    =&d^2J(\Gamma_z;V_1(0);V_2(0))+dJ(\Gamma;W'(0)),
\end{align*}
with $W'(0)(x)$ given by
\begin{align*}
    W'(0)(x)=&\lim_{t\rightarrow0^+}\frac{W(t,x)-W(0,x)}{t}
    =-\frac{\partial}{\partial z_j}\Big[\frac{\nabla_x f_\theta}{\|\nabla_x f_\theta\|^2} \frac{\partial f_\theta}{\partial z_i}\Big](z,x).
\end{align*}
Now it follows from Theorem \ref{hessian} that
\begin{align*}
    \left|\frac{\partial^2 \mathcal{L}}{\partial {z_i} \partial {z_j}}\right|\le&|d^2J(\Gamma_z;V_1(0);V_2(0))+dJ(\Gamma_z;W'(0))|\\ \le&C\|V_1(0)\|_{C^{3}(\Gamma;\mathbb{R}^3)}\|V_2(0)\|_{C^{3}(\Gamma;\mathbb{R}^3)}+C\|W'(0)\|_{C^1(\Gamma)}\\
    \le & C\left\|\frac{\nabla_z f_\theta}{||\nabla_x f_\theta||}\right\|_{C^{3}(\Gamma;\mathbb{R}^3)}^2+C\left\|\frac{\partial}{\partial z_j}\Big[\frac{\nabla_x f_\theta}{\|\nabla_x f_\theta\|^2} \frac{\partial f_\theta}{\partial z_i}\Big]\right\|_{C^1(\Gamma)}.
\end{align*}
Since the label $z$ is bounded, by the smoothness of the DeepSDF $f_\theta(z,x)$, there exists a constant $L$ such that
$\|\nabla^2 \mathcal{L}(z)\|\le L$ for all $z$.
Finally, this inequality and  \cite[Theorem 3.1]{chen2019convergence} imply the desired assertions \eqref{bounded} and \eqref{rate}.
\end{proof}

\section{Numerical experiments and discussions}\label{results}

In this section, we present numerical results to illustrate the performance and distinct features of the proposed approach, including backscattering and phaseless data.

\subsection{The training of the neural network}
We train the DeepSDF following \cite{park2019deepsdf},  with two scenes from dataset Shapenet \cite{chang2015shapenet}, a richly-annotated, large-scale repository of shapes represented by 3D CAD models of objects. In the first scene (Scenario 1), there are 1780 surfaces in the training dataset randomly selected from the Shapenet airplane dataset, and the test data is also randomly selected in the Shapenet airplane dataset (but not in the training dataset). In the second scene (Scenario 2), there are 1258 surfaces in the car dataset. We use the implementation provided by DeepSDF \cite{park2019deepsdf} to train the latent representation of the surfaces. We learn the DeepSDF $f_\theta(z,x)$ using the ADAM algorithm \cite{KingmaBa:2015} with a step decay learning rate schedule $\alpha_n=\text{5.0e-4}/{2^{[n//500]}}$, i.e., the learning rate $\alpha_n$ is halved every 500 epochs. The dimensionality $Z$ of the latent variable $z$ is 256.

To show that the learned shape embedding DeepSDF $f_\theta(z,x)$ is continuous, we render the results of the decoder when a pair of shapes are interpolated in the latent vector space; see Figure \ref{continuous} for an illustration. For two latent $z_1,z_2$, we take the interpolated latent value $z=(z_1+z_2)/2$. The plot shows that the embedded continuous SDF $f_\theta(z,x)$ does produce meaningful shapes and that the latent representation extracts common interpretable shape features, e.g., plane wings, that interpolate linearly in the latent space. This property is needed for solving ISP.

\begin{figure}[htb!]
    \centering
    \includegraphics{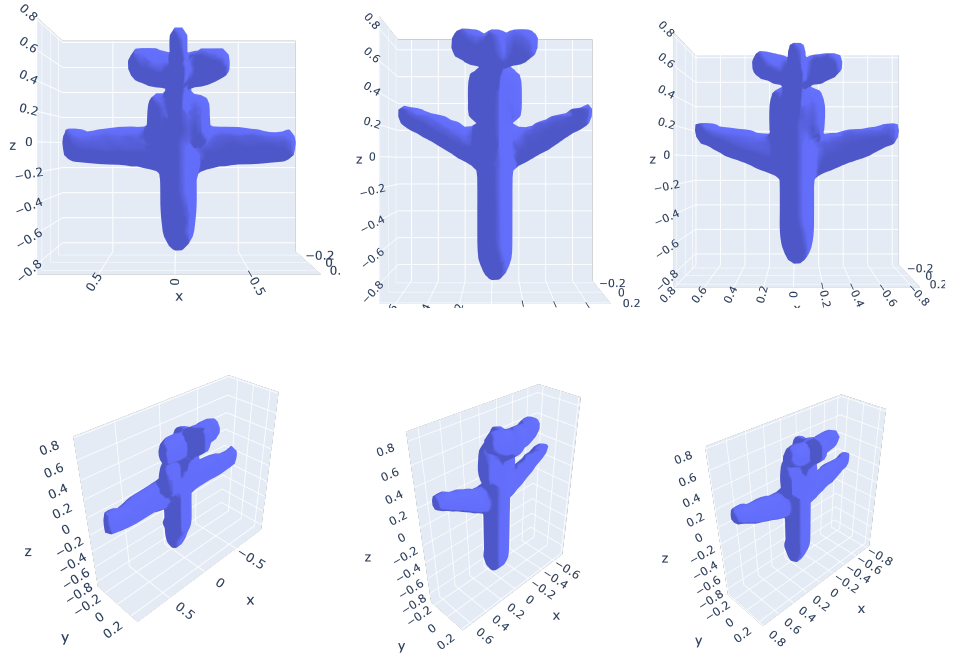}
    \caption{The interpolation between the first and second columns in the learned shape latent space, shown in the third column, from two view angles.}
    \label{continuous}
\end{figure}

\subsection{Numerical examples}
Now we present numerical examples to show the effectiveness of the proposed method. Throughout the numerical experiments, we employ the following experimental setup. We assume that the scatterers are contained in the cube $(-1,1)^3$. With the learned SDF $f_\theta(z,x)$, we use the marching cube algorithm \cite{lorensen1987marching} to generate the surface mesh. To use the marching cube algorithm, the domain $(-1,1)^3$ is discretized using a Cartesian grid with a grid spacing of 0.06. Then we use the boundary element method library BEMPP-CL \cite{betcke2021bempp} to solve the boundary integral equation for problems \eqref{hel} and \eqref{adjoint} to compute the gradient $\nabla \mathcal{L}(z)$ of the loss $\mathcal{L}(z)$ by \eqref{gradient}. We employ the ADAM algorithm implemented in the \texttt{PyTorch.optimizer} package to update the latent variable $z$, with a constant learning rate $\alpha=0.01$ and the dimensionality $Z$ of the latent variable $z$ is 256.

To solve the direct and adjoint problems (for computing the shape derivative), we employ the boundary element method, for which we use a Galerkin formulation of equation \eqref{bie} then approximate the density field on the mesh of the surface with discontinuous Galerkin piecewise constant functions and solve the resulting linear system by the GMRES algorithm with a relative error 1e-5. These choices are used for both generating the data and also in numerical
inversion. The public software package BEMPP \cite{betcke2021bempp} provides the necessary implementations of boundary element spaces, potentials, and boundary operators. For an overview of the library, see
\cite{smigaj2015solving}. The scattering from an impenetrable obstacle can be implemented in BEMPP using a few lines of code. The actual implementation in BEMPP of the combined integral equation \eqref{bie} can be found in the Jupyter notebook on the homepage of BEMPP (https://bempp.com). More details about the specific BEM implementation and algorithmic parameter settings of the proposed latent representation method can be found at the GitHub page of the project \url{https://github.com/liuhaibogit/Latent_ISP}.

Note that in practice, 
we first generate the surface mesh from the latent representation using the marching cubes algorithm\cite{lorensen1987marching}, which is then directly taken to be the boundary element surface mesh. So the number of BEM mesh elements varies with the iteration, which however do not change much throughout the iterations,  and the inverse crime is automatically alleviated. In Figure \ref{grid}, we show one example surface mesh  (which is employed to resolve the direct and adjoint problems in the experiment in scenario 1 below), obtained by using the marching cube algorithm, with the target surface  discretized into a triangular mesh. 
Numerically, it can be verified that the relative error of the forward solver for approximating the far field pattern is about 1\%. 


\begin{figure}[htb!]
    \centering
    \includegraphics[width=.8\textwidth]{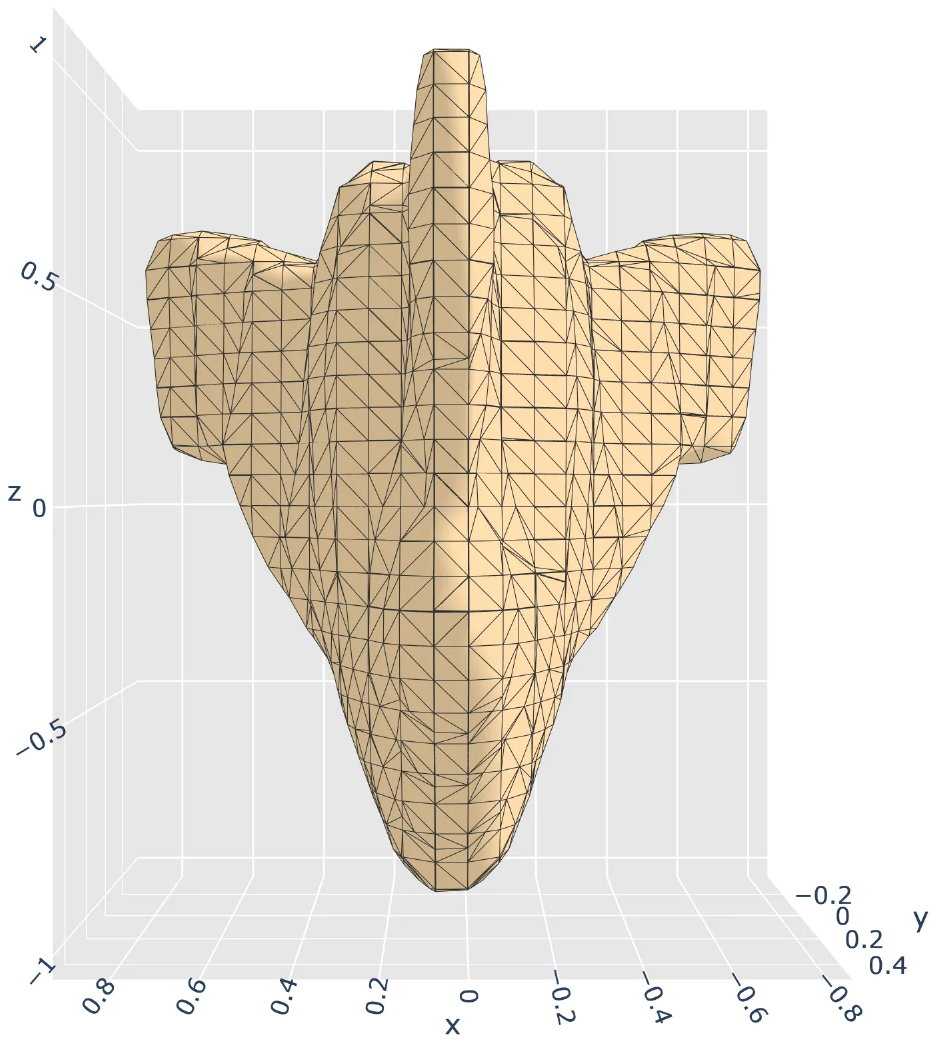}
    \caption{Example triangular mesh with 1190 vertices and 2376 faces.}
    \label{grid}
\end{figure}

The wave number $k$ of the incident wave is chosen to be $k = 5\pi$ unless otherwise stated. For far-field data, we transmit $L$ equidistant incident plane waves on a full aperture and collect measurements at $M$ equidistant angles on a full aperture for each incident wave. 
The equidistant directions for incident waves and measurement on the unit sphere $\mathbb{S}^2$ are generated using the Fibonacci lattice method \cite{gonzalez2010measurement}. Specifically, to generate $N$ equidistant points $\{x_n=(x^1_n,x^2_n,x^2_n)\}_{n=1}^N\in\mathbb{R}^3$ on unit sphere $\mathbb{S}^2$ for $N=L$ or $M$, the points $x_n$ in polar coordinates are given explicitly by
\begin{eqnarray}
    \left\{
    \begin{aligned}
        x^3_n&=(2n-1)/N-1,\\
        x^1_n&=\sqrt{1-(x^3_n)^2} \cdot \cos{2\pi n \phi},\\
        x^2_n&=\sqrt{1-(x_n^3)^2} \cdot \sin{2\pi n \phi},
    \end{aligned}
    \right.\label{fibona}
\end{eqnarray}
where $\phi=(\sqrt{5}+1)/2$ is the golden ratio. 
To generate the exact data $u^\infty_{\Omega^*}(\hat{x}_m,d_l)$ for $\hat{x}_m, d_l\in \mathbb{S}^2$ and $m=1,2,...M, l=1,2,..., L$, we use the BEMPP-CL package to solve problem \eqref{hel} with exact domain $\Omega^*$. 
We also investigate the performance of the proposed method with respect to the noise. We generate the noisy data $u_{ml}^\delta$ by 
$u^\delta_{ml} = (1+\delta \xi_{ml})u^\infty_{\Omega^*}(\hat{x}_m,d_l)$,
where $\xi_{ml}$ follows the Gaussian random variable with zero mean and unity variance, and $\delta>0$ controls the relative noise level. To measure the accuracy of  the reconstructed surface $\hat S$ relative to the exact one $S^\dag$, we employ the indicator error:
\begin{eqnarray*}
    e(\hat S)=\sum_{x_i}\|\hat S(x_i)-S^\dag(x_i)\|_2,
\end{eqnarray*}
where $x_i$s are the grid points on a uniform grid and 
\begin{eqnarray*}
    S(x)=\left\{
    \begin{aligned}
        1,&\quad \text{if $x$ is outside surface $S$},\\
        0,&\quad \text{if $x$ is inside surface $S$}.
    \end{aligned}
    \right.
\end{eqnarray*}

Now we present and discuss the reconstruction results in more detail. Note that we initialize latent variable $z$ by randomly selecting a surface in the training set, which may either result in an initial guess with fine details or one relatively smooth, depending on the value of $z$. The numerical experiments below indicate that the performance of the algorithm is robust with respect to the initialization of $z$.

\subsubsection{Scenario 1}
We choose one sample from the test set about the airplane and present the reconstructions and convergence behavior of the iterative algorithm.  We transmit $L=4$ equidistant incident plane waves on a full aperture and collect measurements at $M=100$ equidistant angles on a full aperture for each incident wave. Figure \ref{intermidiate1} shows the target and intermediate recovery results at different iterations. The recovered surfaces approach the exact target as the iteration of Algorithm \ref{algorithm} proceeds, and the general shape can be recovered in tens of iterations. Figures \ref{example1:result} and \ref{example1:noiseresult} show the reconstructions of a target using exact data and data with a noise level $\delta=40\%$, respectively. The reconstructions are shown in two different view angles. The first column shows the initial guess for Algorithm \ref{algorithm}, the second column the final reconstructions, and the third column the target in order to provide more complete pictures. The convergence of the algorithm versus the iteration number is shown in Figure \ref{example1:error}, in terms of the loss $\mathcal{L}(z)$  and the indicator error $e$. These results indicate that with the ADAM algorithm, the loss $\mathcal{L}(z)$ and the indicator error $e$ decrease rapidly, for both exact and noisy data. The algorithm generally converges stably, but the loss  $\mathcal{L}(z)$ may exhibit some oscillations. The oscillations on the loss curve $\mathcal{L}(z)$ might be due to a too large learning rate (as is the case for the experiment) and numerical errors in computing the gradient (e.g., too coarse mesh, low-order elements and inexact solution of the BEM linear system). The oscillations can be reduced by using a smaller step size, but not completely removed. The presence of oscillations in the loss curve may cause difficulty in terminating the iteration. In the experiment, we have terminated the iteration in an ad-hoc trial and error manner such that the obtained reconstruction is accurate. Furthermore, Algorithm \ref{algorithm} is quite stable with respect to the presence of data noise: even using very noisy data, we can still reconstruct the object successfully. To further illustrate the robustness of the proposed approach, we show in Figures \ref{revise_init} and \ref{revise_init_error} the results with a smooth initial guess. It is worth observing that even when initialized with a surface with a simple structure, the algorithm can still recover the target very well, due to the use of the latent representation. Although not presented, similar observations can be drawn for other examples.

\begin{figure}[htb!]
    \centering
    \includegraphics{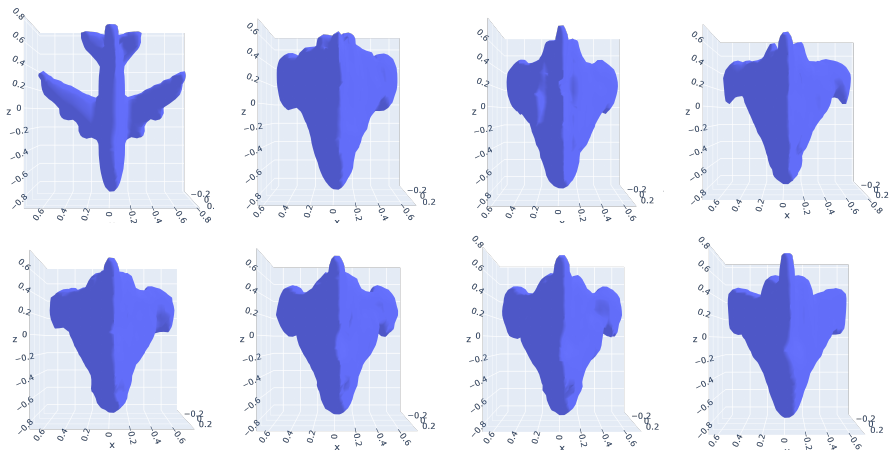}
    \caption{The convergence of the optimization algorithm for Scenario 1: the intermediate reconstruction at iterations: 0, 26, 53, 80, 107, 134, and 161 (ordered from left to right and from top to bottom), and the last plot denotes the exact target.}
    \label{intermidiate1}
\end{figure}

\begin{figure}[htb!]
    \centering
    \includegraphics{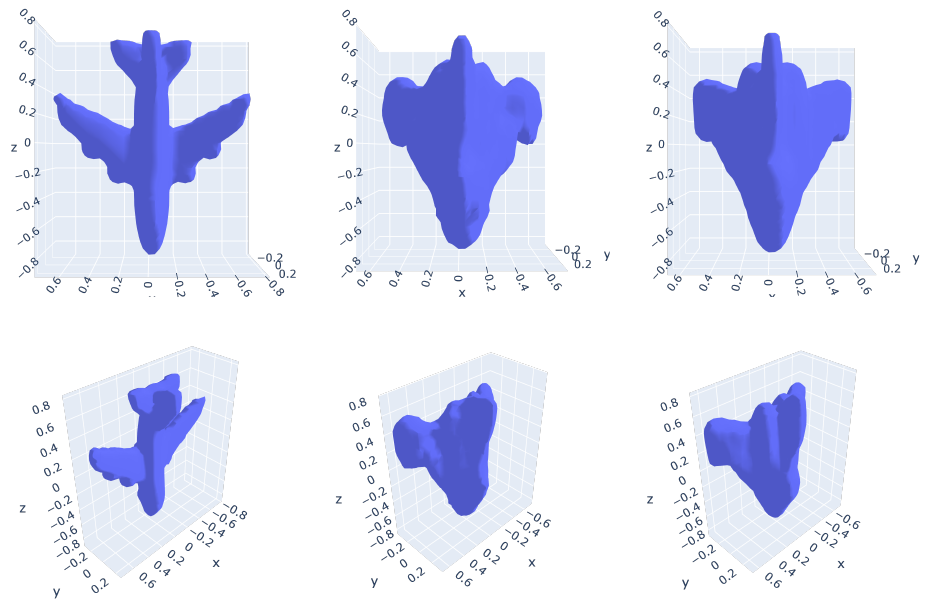}
    \caption{The reconstructions for Scenario 1 with exact data: the left, middle, and right columns refer to the initial, the optimal recovery, and the ground truth target, respectively.}
    \label{example1:result}
\end{figure}

\begin{figure}[htb!]
    \centering
    \includegraphics{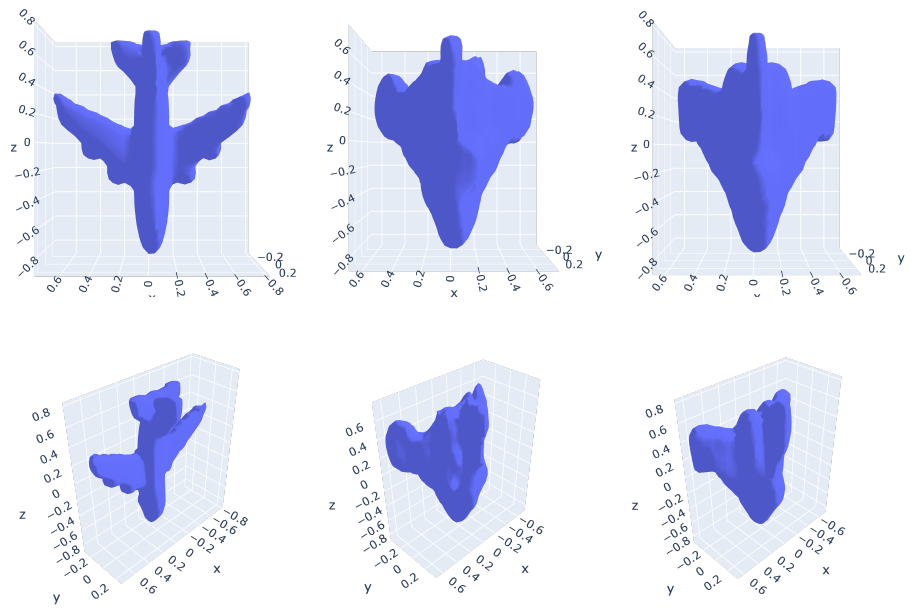}
    \caption{The reconstructions for Scenario 1 with 40\% noise in the data: the left, middle, and right columns refer to the initial, the optimal recovery, and the ground truth target, respectively.}
    \label{example1:noiseresult}
\end{figure}

\begin{figure}[htb!]
    \centering
    \begin{tabular}{cc}
    \includegraphics[width=.47\textwidth]{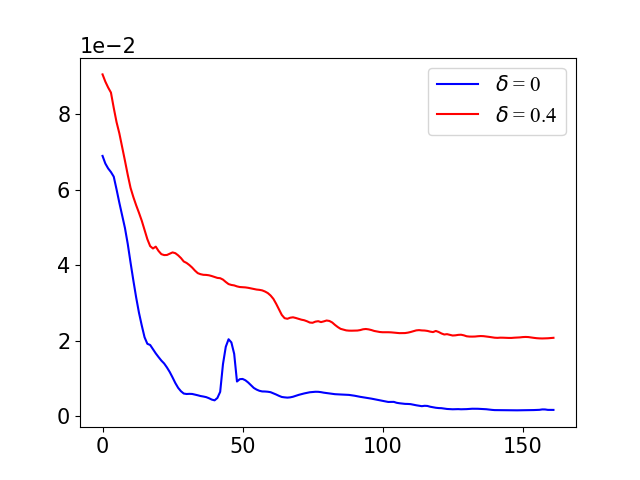}  & \includegraphics[width=.47\textwidth]{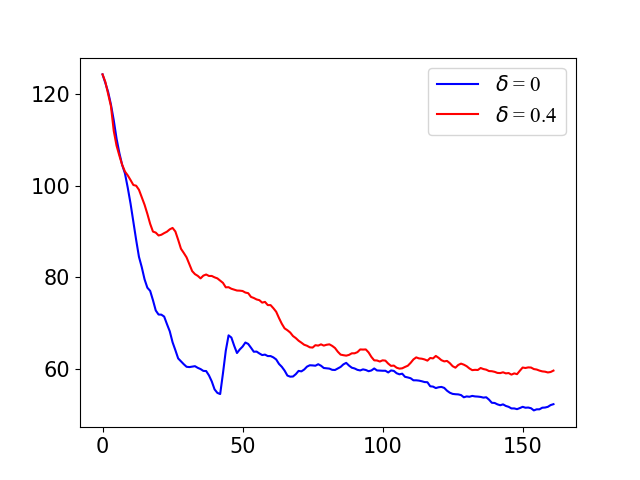} \\
    $\mathcal{L}(z^n)$ v.s. $n$  & $e$ v.s. $n$ \end{tabular}
    \caption{Convergence of the ADAM algorithm for Scenario 1 in terms of the loss $\mathcal{L}(z)$ (left) and the indicator error $e$ (right). }
    \label{example1:error}
\end{figure}

\begin{figure}[htb!]
    \centering
    \includegraphics{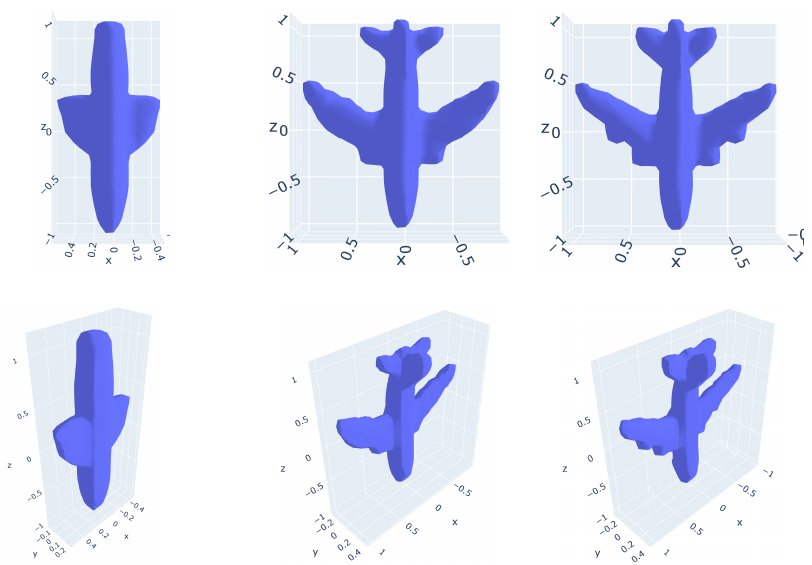}
    \caption{The reconstructions for Scenario 1 with a simple initial surface: the left, middle, and right columns refer to the initial, the optimal recovery, and the ground truth target, respectively.}
    \label{revise_init}
\end{figure}

\begin{figure}[htb!]
    \centering
    \begin{tabular}{cc}
    \includegraphics[width=.48\textwidth]{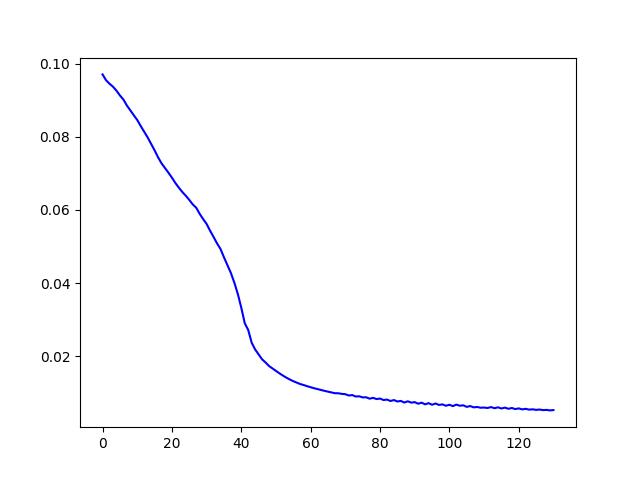}  & \includegraphics[width=.48\textwidth]{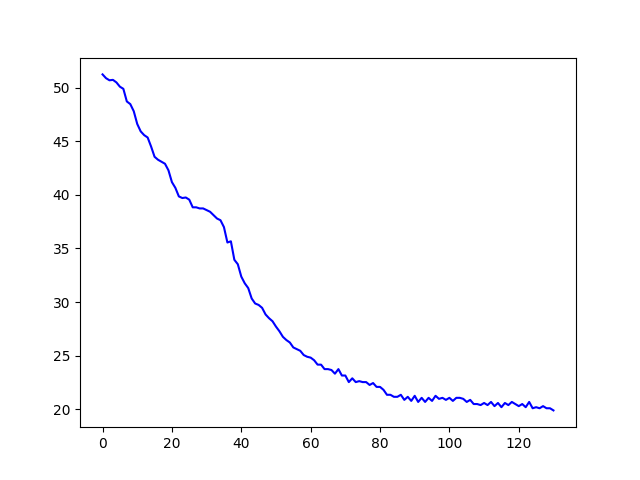} \\
    $\mathcal{L}(z^n)$ v.s. $n$  & $e$ v.s. $n$ \end{tabular}
    \caption{Convergence of the ADAM algorithm for Scenario 1 in terms of the loss $\mathcal{L}(z)$ (left) and the indicator error $e$ (right).}
    \label{revise_init_error}
\end{figure}

\subsubsection{Scenario 2}
Likewise, we also choose one sample from the test set about the car and present the reconstructions and convergence. We transmit $L=4$ equidistant incident plane waves on a full aperture and collect measurements at $M=100$ equidistant angles on a full aperture for each incident wave. Figures \ref{example3:result} and \ref{example3:noiseresult} show the reconstruction of the target using exact and noisy data with 40\% noise, respectively, and Figure \ref{example3:error} for the convergence of the algorithm. The convergence plot shows that the ADAM algorithm can decrease the loss $\mathcal{L}(z)$ and the indicator error $e$ rapidly, and it enjoys remarkable robustness with respect to data noise: even for data with 40\% noise, we can still reconstruct the object successfully.

\begin{figure}[hbt!]
    \centering
    \includegraphics{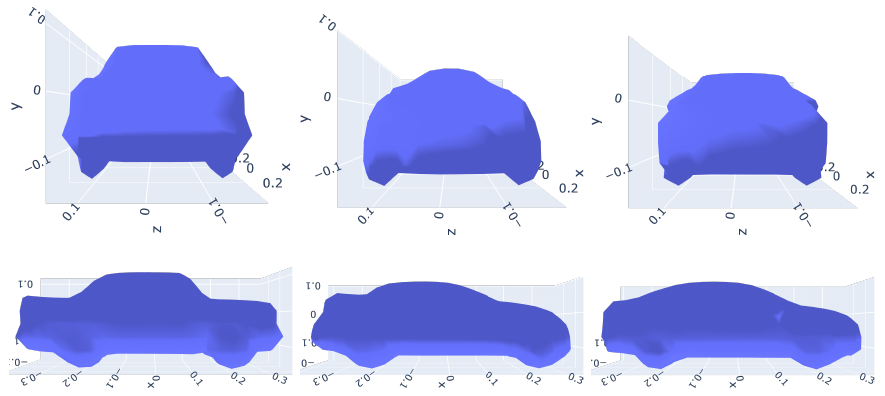}
    \caption{The reconstructions for Scenario 2 with exact data: the left, middle, and right columns refer to the initial, the optimal recovery, and the ground truth target, respectively.}
    \label{example3:result}
\end{figure}

\begin{figure}[hbt!]
    \centering
    \includegraphics{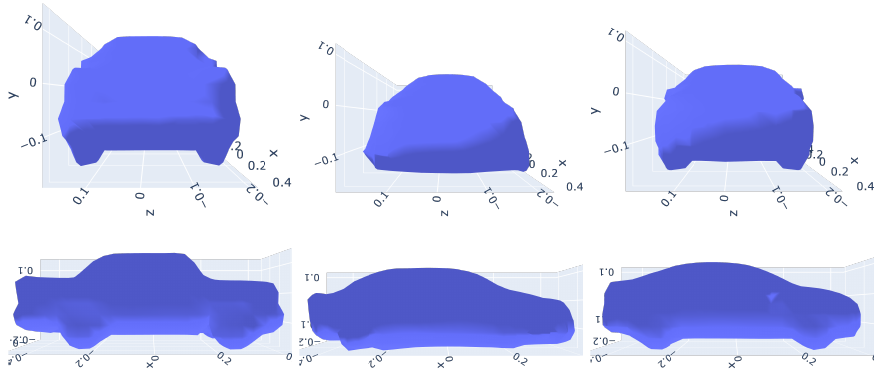}
    \caption{The reconstruction for Scenario 2 with 40\% noise in the data: the left, middle, and right columns refer to the initial, the optimal recovery, and the ground truth target, respectively.}
    \label{example3:noiseresult}
\end{figure}

\begin{figure}[hbt!]
    \centering
    \begin{tabular}{cc}
    \includegraphics[width=.48\textwidth]{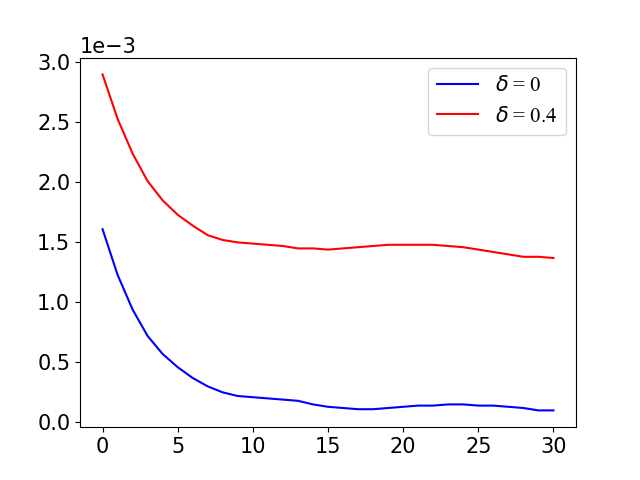}& \includegraphics[width=.48\textwidth]{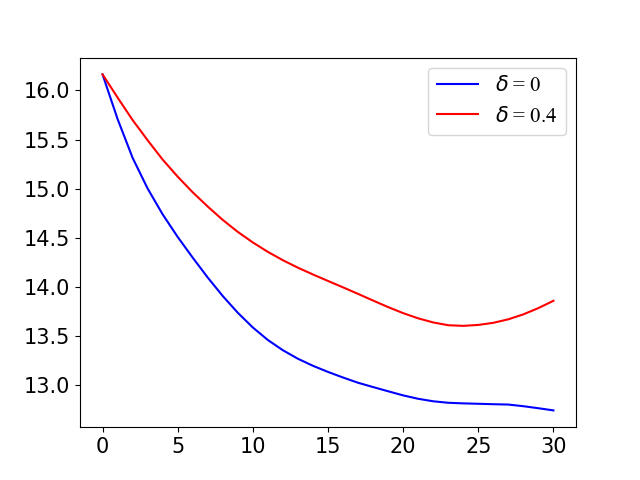}\\
     $\mathcal{L}(z^n)$ v.s. $n$  & $e$ v.s. $n$ \end{tabular}
    \caption{Convergence of the ADAM algorithm for Scenario 2 in terms of the loss $\mathcal{L}(z)$ (left) and the indicator error $e$ (right). }
    \label{example3:error}
\end{figure}

In both scenarios, we observe that the reconstruction errors are highly stable with respect to the noise (up to 40\% noise level), and the recovered shape has a relatively complex structure that simple hand-crafted priors cannot effectively encode. The observed stability property is remarkable for iterative methods for inverse scattering problems since ISP is severely ill-posed. It is attributed to the strong prior imposed by the latent representation.

\subsubsection{Reconstruction with backscattering far field data}
Now we test the proposed algorithm with the backscattering data. Backscattering means that for each incident direction $d$, we only measure the far field data $u_\infty$ at direction $-d$, and hence the available data are severely missing. See \cite{KlibanovKolesovNguyen:2019,DouLiu:2022} for existing reconstruction approaches with backscattering data. With backscattering far field data, the total field $u_l$ for the scattering problem \eqref{hel} with $d = d_l$ and the corresponding adjoint solution $w_l$ of \eqref{adjoint} satisfies
\begin{eqnarray*}
    w_l=\frac{1}{4\pi}\overline{ \Big(u^\infty_{\Omega_z}(-d_l,d_l)-u^\infty_{\Omega^*}(-d_l,d_l)\Big)}u_l.
\end{eqnarray*}
Similar to \eqref{gradient}, the gradient $\nabla\mathcal{L}(z)$ of the loss $\mathcal{L}(z)$ can be directly computed as
\begin{eqnarray*}
    \nabla \mathcal{L}(z)=-\frac{1}{4\pi L}\Re\left(\sum_{l=1}^L\overline{ \Big(u^\infty_{\Omega_z}(-d_l,d_l)-u^\infty_{\Omega^*}(-d_l,d_l)\Big)} \int_{\Gamma_z}\big({\frac{\partial u_l}{\partial \nu}\big)^2 \frac{\nabla_z f_\theta}{\|\nabla_x f_\theta\|} {\rm d}s}\right).
\end{eqnarray*} 
Thus the proposed algorithm still applies. We test the algorithm on Scenario 1. The reconstruction results for 4 equidistant incident plane waves on a full aperture are shown in the first row of Figure {\ref{Back}}. The results only give a rough shape of the target obstacle, but lack sufficient details, especially in the plane's tail, due to a lack of data. Therefore, we increase the number of transmitters to 20, and the results are shown in the second row of Figure {\ref{Back}} and Figure {\ref{Back_error}}. The reconstructions become much more accurate than that forthe case of the four transmitted incident waves. Hence, the algorithm can still succeed when using enough backscattering far-field data. This again shows the effectiveness and robustness of the proposed algorithm.

\begin{figure}[htb!]
    \centering
    \includegraphics{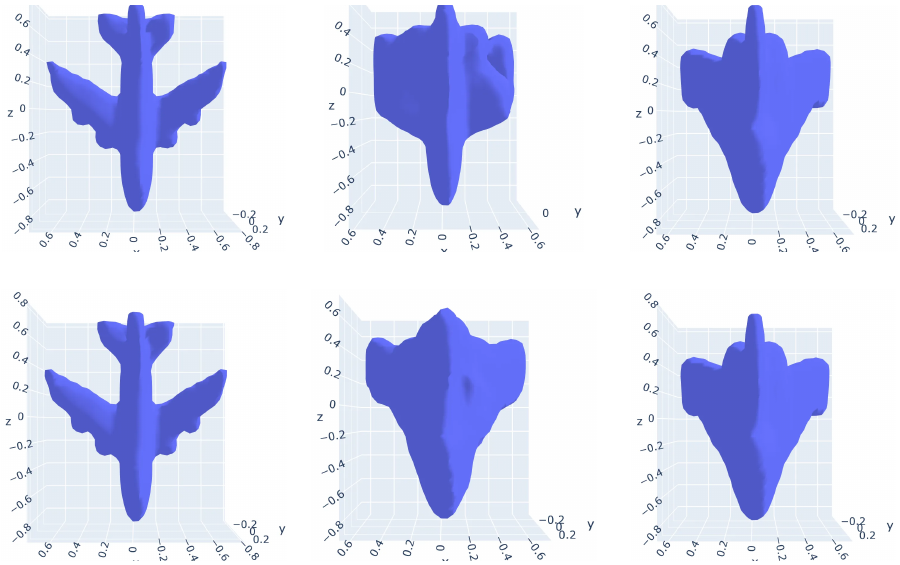}
    \caption{The reconstructions with backscattering far field data: the left, middle, and right columns refer to the initial, the optimal recovery, and the ground truth target, respectively.}
    \label{Back}
\end{figure}

\begin{figure}[hbt!]
    \centering
    \begin{tabular}{cc}
    \includegraphics[width=.48\textwidth]{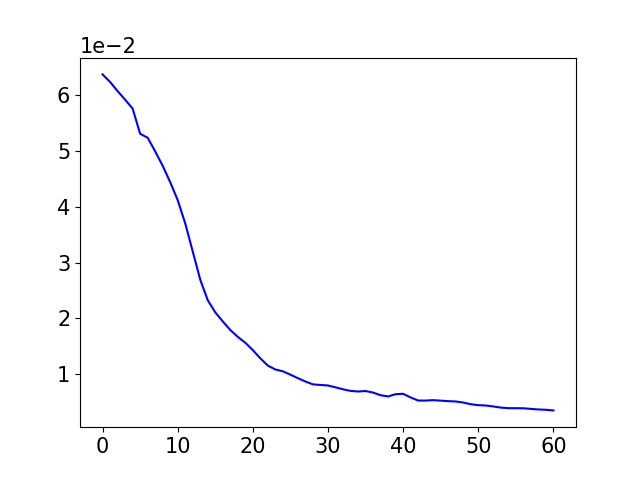}& \includegraphics[width=.48\textwidth]{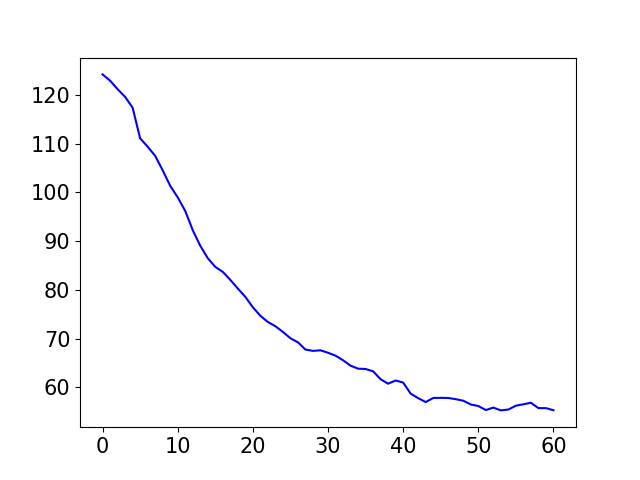}\\
     $\mathcal{L}(z^n)$ v.s. $n$  & $e$ v.s. $n$ \end{tabular}
    \caption{Convergence of the ADAM algorithm with backscattering far field data in terms of the loss $\mathcal{L}(z)$ (left) and the indicator error $e$ (right).}
    \label{Back_error}
\end{figure}

\subsection{Reconstructions with phaseless data}
Last, we illustrate the approach with phaseless data. In many practical applications, the phase of the far-field pattern cannot be measured accurately when compared with its modulus or intensity. Therefore, it is often desirable to reconstruct the scattering obstacles from phaseless data; see the works \cite{KressRundell:1997,ivanyshyn2007shape,Klibanov:2014,Novikov:2015,AmmariChowZou:2016} for uniqueness and reconstruction techniques with phaseless data. Following the work \cite{audibert2022accelerated}, the cost functional $\mathcal{L}(z)$ with phaseless data is taken to be
\begin{eqnarray*}
    \mathcal{L}(z)=J(\Gamma_z)=\frac{1}{2LM} \sum_{l=1}^L \sum_{m=1}^M\Big| \frac{|u^\infty_{\Omega_z}(\hat{x}_m,d_l)|^2}{\sqrt{|u^\infty_{\Omega_z}(\hat{x}_m,d_l)|^2+\epsilon}}-\frac{|u^\infty_{\Omega^\star}(\hat{x}_m,d_l)|^2}{\sqrt{|u^\infty_{\Omega^\star}(\hat{x}_m,d_l)|^2+\epsilon}}\Big|^2,
\end{eqnarray*}
where $\epsilon>0$ is a small number to ensure the differentiability.
The gradient can be calculated by
\eqref{gradient} and the incident wave $w_i$ for the adjoint problem \eqref{adjoint} is
\begin{align*}
    w_l^i(y)=&\frac{1}{4\pi M}\sum_{m=1}^{M} \left(\frac{|u^\infty_{\Omega_z}(\hat{x}_m,d_l)|^2}{\sqrt{|u^\infty_{\Omega_z}(\hat{x}_m,d_l)|^2+\epsilon}}-\frac{|u^\infty_{\Omega^\star}(\hat{x}_m,d_l)|^2}{\sqrt{|u^\infty_{\Omega^\star}(\hat{x}_m,d_l)|^2+\epsilon}}\right)\\
    &\times\frac{|u^\infty_{\Omega_z}(\hat{x}_m,d_l)|^2+2\epsilon}{(|u^\infty_{\Omega_z}(\hat{x}_m,d_l)|^2+\epsilon)^{3/2}} \overline{u^\infty_{\Omega_z}(\hat{x}_m,d_l)}  e^{-ik\hat{x}_m \cdot y}.
\end{align*}
The reconstruction algorithm is almost the same except the cost function and adjoint field $w_l$, and we test Algorithm \ref{algorithm} for
Scenario 1.  We use 10 equidistant transmitters and collect phaseless measurements at 300 equidistant angles on a full aperture for each incident wave. Due to the translation invariance of phaseless data, one can only reconstruct the shapes up to a translation (which is the same for all connected components). The reconstructed results are shown in Figures {\ref{phase} and \ref{phase:error}}. It is observed that the algorithm can still recover the boundary well.

\begin{figure}[htb!]
    \centering
    \includegraphics{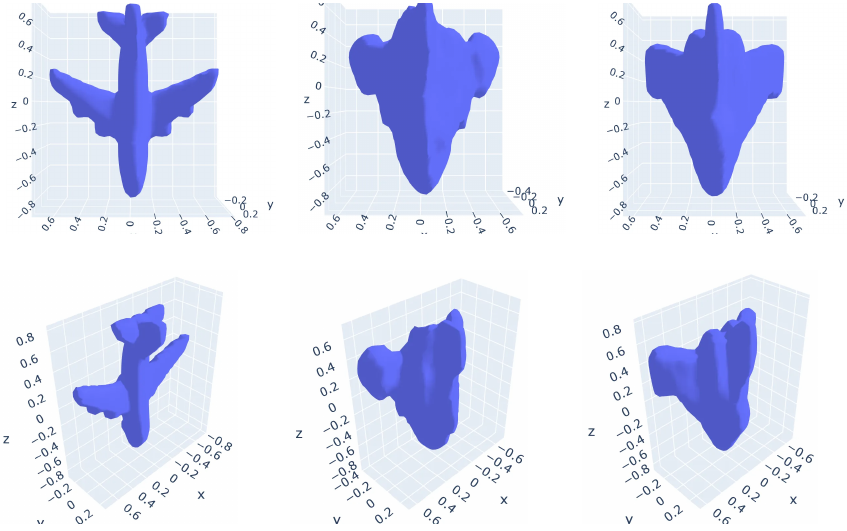}
    \caption{The reconstructions with phaseless data: the left, middle, and right columns refer to the initial, the optimal recovery, and the ground truth target, respectively.}
    \label{phase}
\end{figure}

\begin{figure}[hbt!]
    \centering
    \begin{tabular}{cc}
    \includegraphics[width=.48\textwidth]{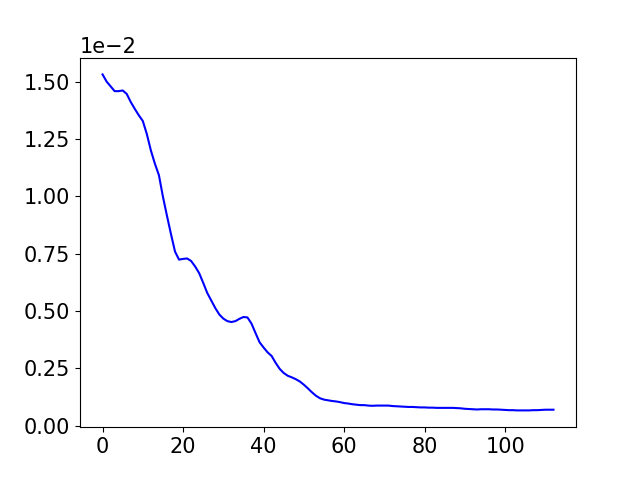}& \includegraphics[width=.48\textwidth]{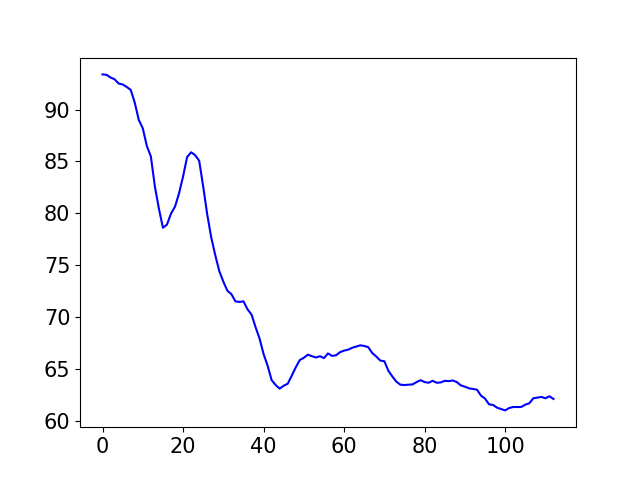}\\
     $\mathcal{L}(z^n)$ v.s. $n$  & $e$ v.s. $n$ \end{tabular}
    \caption{Convergence of the ADAM algorithm with phaseless data in terms of the loss $\mathcal{L}(z)$ (left) and the indicator error $e$ (right).}
    \label{phase:error}
\end{figure}

\section{Conclusion}\label{conclusion}

Based on the latent representation of surfaces, we have developed an efficient iterative method to solve the inverse obstacle scattering problem. The latent representation using well-trained neural networks can significantly reduce the dimensionality of the parameter space.  We derive rigorously the gradient of the loss using shape derivative, which can be evaluated via the adjoint method and automatic differentiation in Pytorch. With the gradient at hand, we can employ the ADAM algorithm to minimize the objective function. We analyzed the convergence of the algorithm by establishing the Lipschitz property of the gradient of the loss. Two numerical experiments show the high efficiency of the proposed method. The numerical results show that the reconstruction errors are highly robust with respect to the noise (up to 40\%), even though it is extremely ill-posed and sensitive to noise. In sum, the proposed algorithm has considerable potential for solving the inverse scattering problem in specific application domains.

There are several avenues for further research. First it is of much interest to apply the proposed method to other scenarios, e.g., time-dependent wave scattering and high-frequency wave scattering. Conceptually the proposed method applies equally well, since the latent surface representation is independent of the forward map. However, in these more challenging scenarios, the associated forward and adjoint problems are computationally much more demanding to resolve. Further, to fully harness the high-frequency data for high-resolution reconstructions, it is also crucial to have a high-resolution yet expressive latent space representation, which calls for abundant high-quality training data. Second, it is of much interest to analyze theoretical properties of the proposed method, including sample complexity analysis of learning the latent space representation (the prior), regularizing property of the regularized reconstruction and the dynamics of training etc. Third and last, the case of multiple inclusions deserves further investigations, especially when the number of components is unknown and the inclusions are multiscale in nature (i.e., the inclusions of different sizes).

\begin{appendices}
\section{Preliminaries on fractional-order Sobolev spaces} 

In this appendix, we collect several preliminary results on fractional order Sobolev spaces. Recall that for $s\in\mathbb{R}_+$, with $s=m+\sigma$, $m\in\mathbb{N}\cup \{0\}$ and $\sigma\in(0,1)$, the norm $H^s(\Gamma)$ is defined by 
\begin{align*}
        \|u\|_{H^{s}(\Gamma)}^2=&\|u\|_{H^{m}(\Gamma)}^2+\sum_{|\alpha|=m}\int_{\Gamma}\int_{\Gamma}\frac{|D^{\alpha} u(x)-D^{\alpha}u(y)|^2}{|x-y|^{(2+2\sigma)}}{\rm d}s(x){\rm d}s(y),
\end{align*}
via the so-called Sobolev-slobodeckij seminorm.

\begin{lemma}\label{fg}
For $f,g\in H^{s}(\Gamma)\cap C^{\lfloor s \rfloor}(\Gamma), s>0\in \mathbb{R} \backslash \mathbb{N}$, there holds
\begin{align*}
    \|fg\|_{H^{s}(\Gamma)}\le C\big(\|f\|_{C^{\lfloor s \rfloor}(\Gamma)}\|g\|_{H^{s}(\Gamma)}+\|g\|_{C^{\lfloor s \rfloor}(\Gamma)}\|f\|_{H^{s}(\Gamma)}+\|g\|_{C^{\lfloor s \rfloor}(\Gamma)}\|f\|_{C^{\lfloor s \rfloor}(\Gamma)}\big).
\end{align*}
\end{lemma}
\begin{proof}
If $0<s<1$, by definition, 
\begin{align*}
    \|fg\|_{H^{s}(\Gamma)}^2=\|fg\|_{L^{2}(\Gamma)}^2+\int_{\Gamma}\int_{\Gamma}\frac{|f(x)g(x)-f(y)g(y)|^2}{|x-y|^{2+2s}}{\rm d}s(x){\rm d}s(y).
\end{align*}
We denote the second term ($H^s(\Gamma)$ semi-norm) by ${\rm I}$. Then the identity $f(x)g(x)-f(y)g(y)=f(x)(g(x)-g(y))+(f(x)-f(y))g(y)$ and the triangle inequality gives
\begin{align*}
    {\rm I} \le&C\|f\|_{C(\Gamma)}^2\int_{\Gamma}\int_{\Gamma}\frac{|g(x)-g(y)|^2}{|x-y|^{(2+2s)}}{\rm d}s(x){\rm d}s(y)+C\|g\|_{C(\Gamma)}^2\int_{\Gamma}\int_{\Gamma}\frac{|f(x)-f(y)|^2}{|x-y|^{(2+2s)}}{\rm d}s(x){\rm d}s(y)\\
    =&C\|f\|_{C(\Gamma)}^2(\|g\|_{H^{s}(\Gamma)}^2-\|g\|_{L^{2}(\Gamma)}^2)+C\|g\|_{C(\Gamma)}^2(\|f\|_{H^{s}(\Gamma)}^2-\|f\|_{L^{2}(\Gamma)}^2).
\end{align*}
This implies 
the desired assertion for $s\in(0,1)$.
When $1<s<2$, let  $t=s-1$. Then the inequality for $s\in(0,1)$ implies
\begin{align*}
    &\|fg\|_{H^{s}(\Gamma)}
    \le C\Big(\|fg\|_{H^{1}(\Gamma)}+\sum_{|\alpha|=1}\|gD^\alpha f\|_{H^{t}(\Gamma)}+\sum_{|\alpha|=1}\|fD^\alpha g\|_{H^{t}(\Gamma)}\Big)\\
    \le&C\Big(\|fg\|_{H^{1}(\Gamma)}+\|g\|_{C^{1}(\Gamma)}\sum_{|\alpha|=1}\|D^\alpha f\|_{H^{t}(\Gamma)}+\|g\|_{H^{t}(\Gamma)}\sum_{|\alpha|=1}\|D^\alpha f\|_{C^{1}(\Gamma)}+\|g\|_{C^{1}(\Gamma)}\sum_{|\alpha|=1}\|D^\alpha f\|_{C^{1}(\Gamma)}\\
    &+\|f\|_{H^{t}(\Gamma)}\sum_{|\alpha|=1}\|D^\alpha g\|_{C^{1}(\Gamma)}+\|f\|_{C^{1}(\Gamma)}\sum_{|\alpha|=1}\|D^\alpha g\|_{H^{t}(\Gamma)}+\|D^\alpha g\|_{C^{1}(\Gamma)}\sum_{|\alpha|=1}\|f\|_{C^{1}(\Gamma)}\Big)\\
    \le&C\big(\|f\|_{C^{1}(\Gamma)}\|g\|_{C^{1}(\Gamma)}+\|f\|_{H^{s}(\Gamma)}\|g\|_{C^{1}(\Gamma)}+\|f\|_{C^{1}(\Gamma)}\|g\|_{H^{s}(\Gamma)}\big).
\end{align*}
The case for $s>2$ can be proved similarly.
\end{proof}

\begin{lemma}\label{fraction}
If $s=m+\sigma,s\in\mathbb{R}_+,\, m\in \mathbb{N}\cup\{0\}$, $0<\sigma<1$, $u\in C^{m+1}(\Gamma)$, then
\begin{eqnarray*}
    \|u\|_{H^{s}(\Gamma)}\le C\|u\|_{H^{m+1}(\Gamma)}\quad \mbox{and} \quad
    \|u\|_{H^{m+1}(\Gamma)} \le C\|u\|_{C^{m+1}(\Gamma)}.
\end{eqnarray*}
\end{lemma}
\begin{proof}
By the definition of fractional-order Sobolev spaces, from \cite[Proposition 2.2]{di2012hitchhikers}, we have
\begin{align*}
    \|u\|_{H^{s}(\Gamma)}^2
    =&\|u\|_{H^{m}(\Gamma)}^2+\sum_{|\alpha|=m}\|D^{\alpha} u\|_{H^{\sigma}(\Gamma)}^2-\sum_{|\alpha|=m}\|D^{\alpha} u\|_{L^{2}(\Gamma)}^2\\
    \le&\|u\|_{H^{m}(\Gamma)}^2+C\sum_{|\alpha|=m}\|D^{\alpha} u\|_{H^{1}(\Gamma)}^2-\sum_{|\alpha|=m}\|D^{\alpha} u\|_{L^{2}(\Gamma)}^2
    \leq C\|u\|_{H^{m+1}(\Gamma)}^2.
\end{align*}
Furthermore,
\begin{align*}
    \|u\|_{H^{m+1}(\Gamma)}^2=&C\sum_{|\alpha|\le m+1}\|D^\alpha u\|_{L^2(\Gamma)}^2 
    \le C\|u\|_{C^{m+1}(\Gamma)}^2.
\end{align*}
This completes the proof of the lemma.
\end{proof}

\end{appendices}

\bibliographystyle{abbrv}
\bibliography{ref}
\end{document}